\documentclass[11pt]{article}

\usepackage[T1]{fontenc}
\usepackage{lmodern}
\usepackage[margin=1in]{geometry}
\usepackage[numbers,square]{natbib}
\usepackage{amsmath,amssymb,amsfonts}
\usepackage{amsthm}
\usepackage{mathrsfs}
\usepackage{algorithm}
\usepackage{algorithmicx}
\usepackage{algpseudocode}
\usepackage[colorlinks=true,linkcolor=blue,citecolor=blue,urlcolor=blue]{hyperref}

\newcommand{\R}{\mathbb{R}}
\newcommand{\N}{\mathbb{N}}
\newcommand{\setC}{\mathcal{C}}
\newcommand{\setD}{\mathcal{D}}
\newcommand{\setU}{\mathcal{U}}
\newcommand{\argmin}{\mathrm{argmin}}
\newcommand{\dom}{\mathrm{dom}}

\newcommand{\dist}{\mathrm{dist}}
\newcommand{\ri}{\mathrm{ri}}
\newcommand{\prox}{\mathrm{prox}}
\newcommand{\AssumpA}{\textbf{\hyperref[assump:A]{Assumption A}}}
\newcommand{\AssumpB}{\textbf{\hyperref[assump:B]{Assumption B}}}
\newcommand{\Hone}{\textbf{(H1)}}
\newcommand{\Htwo}{\textbf{(H2)}}
\newcommand{\Hthree}{\textbf{(H3)}}

\theoremstyle{plain}
\newtheorem{theorem}{Theorem}[section]
\newtheorem{proposition}[theorem]{Proposition}
\newtheorem{lemma}[theorem]{Lemma}
\newtheorem{corollary}[theorem]{Corollary}

\theoremstyle{remark}
\newtheorem{example}{Example}
\newtheorem{remark}{Remark}

\theoremstyle{definition}
\newtheorem{definition}{Definition}

\title{On the Convergence Analysis of DCA}
\author{Yi-Shuai Niu\\
Beijing Institute of Mathematical Sciences and Applications (BIMSA), Beijing, China\\
\texttt{niuyishuai@bimsa.cn}}
\date{}

\begin{document}

\maketitle

\begin{abstract}
Difference-of-Convex (DC) programming, which seeks to minimize a function expressed as the difference of two convex functions, arises in a wide range of applications in machine learning, signal processing, and operations research. A classical and widely used algorithm for solving DC programs is the Difference-of-Convex Algorithm (DCA). In this paper, we revisit DCA from a distinctly DC-specific perspective. We first separate well-definedness from asymptotic convergence and introduce an additional assumption ensuring the solvability of the DCA subproblems, which clarifies why the choice of DC decomposition matters. We then develop a Lyapunov--descent--regularity framework in which the descent estimate is read directly from the convex subproblems and the regularity estimate is verified from DCA optimality conditions. This yields global convergence of the iterates $\{x^k\}$ for both standard and convex-constrained DC programs under either the classical {\L}ojasiewicz subgradient inequality or the broader Kurdyka--{\L}ojasiewicz (KL) property. We further explain how stronger regularity regimes, such as the Polyak--{\L}ojasiewicz (PL) condition, fit into the same framework and sharpen the resulting convergence rates. Consequently, we obtain finite-time, linear, and sublinear rates for objective values and iterates in a way that cleanly separates well-definedness, DCA-specific structure, and KL/PL regularity, and that is readily transferable to DCA-type variants.
\end{abstract}

\noindent\textbf{MSC2020 subject classifications.} 90C26, 90C30, 65K05.


\tableofcontents

\bigskip
\section{Introduction}
\label{intro}
Consider the Convex Constrained Difference-of-Convex (DC) Program defined by 
\begin{equation*}\label{prob:P}
	\min \{ g(x)-h(x): x\in \setC \}\tag{P}
\end{equation*}
under the classical assumptions: \\
\noindent
\fbox{
	\parbox{\dimexpr\textwidth-2\fboxsep-2\fboxrule\relax}{
		\noindent (\AssumpA)\label{assump:A}
		\begin{itemize}
			\item Both $g$ and $h$ belong to $\Gamma_0(\R^n)$, which denotes the set of proper, closed (or lower semi-continuous), and convex functions from $\R^n$ to $\R\cup\{+\infty\}$;
			\item $\setC$ is a nonempty closed and convex set of $\R^n$ (could be identical to $\R^n$).
			\item The solution set of \eqref{prob:P} is non-empty, which implies that $\emptyset \neq (\dom g\cap \setC)\subset \dom h$, where $\dom g:=\{x\in \R^n: g(x)<\infty\}$ is the effective domain of $g$.
		\end{itemize}
	}
}
We adopt the convention $\infty - \infty = \infty$. The standard DC program $\min\{g(x)-h(x):x\in \R^n\}$ can be considered as a special case of \eqref{prob:P} by setting $\setC = \R^n$. Note that
any problem \eqref{prob:P} is equivalent to a standard DC program:
\begin{equation*}\label{prob:P'}
	\min \{ (g+\chi_{\setC})(x)-h(x): x\in \R^n \},\tag{P'}
\end{equation*}
where $\chi_{\setC}$ is the indicator function of $\setC$ defined by $\chi_{\setC}(x)=0$ if $x\in \setC$ and $\infty$ otherwise. Clearly, both $g+\chi_{\setC}$ and $h$ belong to $\Gamma_0(\R^n)$.

DC program is an active field in nonconvex and nonsmooth optimization \cite{Paper_Hartman1959,tao1986algorithms,hiriart1985generalized,Pham1997,Horst1999,DCA30,de2020abc}. The most renowned algorithm for DC program is DCA (described in Algorithm \ref{alg:DCA}), introduced by Pham Dinh Tao in $1985$ as an extension of the subgradient method \cite{tao1986algorithms}, and extensively developed by Le Thi Hoai An and Pham Dinh Tao since $1994$ (see e.g., \cite{Pham1997,Pham1998,Pham2005,le2018convergence,DCA30}). 
\begin{algorithm}
	\caption{DCA for problem \eqref{prob:P}}
	\label{alg:DCA}
	\begin{algorithmic}[1]
		\Require Initial point $x^0\in \dom \partial h$; 
		\For{$k=0,1,\ldots$}
		\State{Compute $y^k\in \partial h(x^k)$;}
		\State{Solve the convex subproblem:
			\begin{equation*}\label{prob:Pk}
				x^{k+1}\in \argmin \{ g(x) - \langle y^k,x \rangle : x\in \setC \}, \tag{P$_k$}
			\end{equation*}
		}
		\EndFor
	\end{algorithmic}
\end{algorithm}

The basic idea of DCA for problem \eqref{prob:P} is to replace the concave part $-h$ by its affine upper model at the current iterate $x^k$, thereby obtaining a convex majorant of the objective $f$ over $\setC$; minimizing this majorant gives the next iterate $x^{k+1}$. This idea coincides with the Majorization-Minimization (MM) algorithm proposed by Hunter and Lange in \cite{hunter2000optimization}, but DCA is distinguished by the fact that the surrogate is generated from a specific DC decomposition.

DCA has proven to be a promising approach in many large-scale real-world applications, such as sparsity learning \cite{Jackl12}, clustering \cite{Lethi2007new}, molecular conformation \cite{Lethi2003large}, portfolio optimization \cite{pham2011efficient,pham2016dc,niu2019higher}, bilinear matrix inequality \cite{niu2014dc}, natural language processing \cite{niu2021difference}, image denoising \cite{RInDCA}, trust region subproblem \cite{Pham1998}, mixed-integer optimization \cite{le2001continuous,niu2008dc,niu2010programmation,niu2019parallel} and eigenvalue complementarity problems \cite{le2012dc,niu2013efficient,niu2015solving,niu2019improved}, to name a few. There are several variants of DCA including the proximal DCA and the linearized proximal DCA \cite{souza2016global,pang2017computing,wen2018proximal,lu2019enhanced}, the Boosted DCA  \cite{BDCA_S,niu2019higher}, the inertial DCA \cite{de2019inertial,RInDCA}, the ADCA \cite{nhat2018accelerated} and the Stochastic DCA \cite{lethi2022stochastic}. These methods incorporate several optimization techniques (e.g., regularization via proximal term, prox-linearization, backtracking line search, Nesterov's extrapolation, and stochastic approximation) to enhance the overall performance of DCA for structured DC optimization problems. Note that some classical convex/nonconvex optimization algorithms can be viewed as DCA with a special DC decomposition \cite{DCA30}, such as the Goldstein-Levitin-Polyak projection algorithm, the proximal point algorithm, the Expectation-Maximization algorithm, the concave-convex procedure, the iterative shrinkage-thresholding algorithm, and the forward-backward splitting algorithm. See \cite{DCA30} and the references therein for a comprehensive introduction about DC program, DCA and applications. 

Another useful viewpoint is the link between DCA and proximal subgradient methods. After adding and subtracting a quadratic regularizer, one can rewrite the DCA subproblem as a proximal step for a regularized convex part driven by a subgradient of the regularized concave part. This does not identify the two schemes completely, because the DC decomposition and the choice of subgradients remain essential in DCA, but it explains why error-bound and PL-type ideas from the proximal subgradient literature can also be informative here; see, e.g., \cite{zhu2023subgradient}.

\paragraph*{Motivations}

DCA may not be well-defined if an inappropriate DC decomposition is used. \AssumpA\ alone does not guarantee the well-definedness of DCA; see Example \ref{eg:01}. This issue is genuinely DC-specific: the solvability of the convex subproblems depends on the decomposition, the admissibility of the next iterate depends on $\dom \partial h$, and for nonsmooth $h$ the choice of $y^k\in\partial h(x^k)$ can affect the asymptotic behavior. As far as we know, these prerequisites are not isolated systematically in the existing literature, although they should be checked before any KL-type argument is invoked.

A second motivation is to clarify what is specific to DCA, as opposed to what belongs to the abstract KL convergence machinery of descent methods \cite{attouch2013convergence}. In DCA, the descent estimate comes from convex subproblem optimality, the residual estimate depends on the decomposition and on the subgradient selection, and DC criticality is weaker than the usual stationarity notions. Our goal is therefore not merely to restate a standard KL template, but to show how a DCA-native proof route can be organized once well-definedness has been separated from convergence.

\paragraph*{Contributions} 
\begin{itemize}
	\item In \S\ref{sec:welldefofDCA}, we identify algorithmic conditions ensuring that DCA is well-defined. In particular, \AssumpB\ separates subproblem solvability from asymptotic convergence arguments and makes explicit the role played by the DC decomposition. The counterexample in Example \ref{eg:01} and the alternative decomposition in Remark \ref{rmk:1} show that this separation is essential rather than cosmetic.
	
	\item In \S\ref{sec:classicalcvofdca}, we intentionally keep the discussion preparatory: it consolidates the classical monotonicity, descent, square-summability, and subsequential criticality properties into a compact toolbox whose formulas are reused later. This shortens subsequent proofs for the unified framework and makes clear which ingredients are classical and which are new.
	
	\item In \S\ref{sec:cvofdcawithL}, we develop a DCA-specific Lyapunov--descent--regularity framework based on \Hone--\Hthree. It clarifies the relation with classical KL-based frameworks, shows how the abstract regularity condition can be verified directly from DCA optimality conditions, and yields global convergence for both standard and convex-constrained DC programs under the {\L}ojasiewicz subgradient inequality or the KL property. We also add a PL specialization and explain how sharper, structure-dependent rates follow as soon as an explicit KL exponent or a PL/error-bound constant is available. These features make the framework directly portable to proximal, linearized, inertial, and other DCA-type variants.
\end{itemize}

Our main contribution, therefore, is a DC-specific convergence framework that separates well-definedness, decomposition effects, and KL/PL regularity while remaining compatible with the broader KL theory for descent methods.

\section{Well-definedness of DCA}\label{sec:welldefofDCA}

The well-definedness of DCA is a crucial question and highly dependent on the intrinsic nature of the DC decomposition. It is important to note that \AssumpA\ alone \textbf{is not sufficient to guarantee the well-definedness of DCA, and DC criticality ($\partial g(x^\star)\cap \partial h(x^\star)\neq \emptyset$) is not a necessary optimality condition for a local minimizer $x^\star$.} To illustrate this, let us consider the following counterexample:

\begin{example}\label{eg:01}
	Consider the problem 
	\begin{equation}\label{prob:eg01}
		\min_{x\in \R_+}\left\{\frac{x^2}{2} + \sqrt{x}\right\},
	\end{equation}
	where the minimizer is $x^\star = 0$. This problem can be formulated as a DC program with the standard DC decomposition
	\[
	\min_{x\in \R}\{ g(x) - h(x)\},
	\]
	where 
	\begin{equation}\label{eq:dcd01_eg01}
		g(x) = \frac{x^2}{2} + \chi_{\R_+}(x), \quad \text{and} \quad  h(x) = 
		\begin{cases}
			-\sqrt{x}, & \text{if } x\geq 0,\\
			+\infty, & \text{otherwise.}
		\end{cases}
	\end{equation}
	Here, \AssumpA\ is satisfied since both $g$ and $h$ belong to $\Gamma_0(\R)$, with $\dom g = \dom h = \R_+$. However, $\dom \partial h = (0,\infty)$ and $\partial h(0) = \emptyset$. Applying DCA with $x^0 = 1$, we have $y^0 \in \partial h(x^0) = \{h'(1)\} = \{-1/2\}$. Then, $x^1 = \argmin \left\{x^2/2 + \chi_{\R_+}(x) - y^0 x\right\} = \{0\}$, which is not included in $\dom \partial h$. As a result, $y^1$ cannot be generated by DCA, making DCA not well-defined in this case.
	
	This example demonstrates that even with a DC decomposition where $\emptyset\neq \dom g = \dom h$, DCA may still fail to be well-defined. Furthermore, the minimizer $x^\star = 0$ is not a DC critical point since $\partial h(0) = \emptyset$. Therefore, DC criticality, defined by $\partial g(x^\star)\cap \partial h(x^\star) \neq \emptyset$, is not a necessary optimality condition for \eqref{prob:P}, as $\partial h(x^\star)$ or $\partial g(x^\star)$ could be empty at a local minimizer $x^\star$.
\end{example}

\begin{remark}\label{rmk:1}
	One may wonder if there is a DC representation for problem \eqref{prob:eg01} such that DCA is well-defined. The answer is YES. Consider the following DC decomposition:
	\begin{equation}
		\label{eq:dcd02_eg01}
		g(x) = \begin{cases}
			\frac{x^2}{2} - \sqrt{x}, & \text{if } x\geq 0,\\
			\infty, & \text{otherwise},
		\end{cases} \quad \text{and} \quad h(x) = \begin{cases}
			-2\sqrt{x}, & \text{if } x\geq 0,\\
			\infty, & \text{otherwise}.
		\end{cases}
	\end{equation}
	Again, \AssumpA\ is satisfied, and both $g$ and $h$ belong to $\Gamma_0(\R)$, with $\dom g = \dom h = \R_+$, $\dom \partial h = (0,\infty)$, and $\partial h(0) = \emptyset$. Applying DCA with any point \( x^k > 0 \), we have 
	\( y^k \in \partial h(x^k) = \{-1/\sqrt{x^k}\} \). Solving for \( x^{k+1} = \argmin \left\{ \frac{x^2}{2} - \sqrt{x} - y^k x : x\geq 0 \right\} \) using Cardano's formula, we obtain the update rule for \( x^{k+1} \) as follows:
	\begin{equation}\label{eq:seq_eg01}
		x^{k+1} = \left( \sqrt[3]{\frac{1}{4} + \sqrt{\Delta_k}} + \sqrt[3]{\frac{1}{4} - \sqrt{\Delta_k}} \right)^2 > 0,
	\end{equation}
	where \( \Delta_k := \frac{1}{16} + \frac{1}{(3\sqrt{x^k})^{3}} > 0 \) whenever \( x^k > 0 \). It is straightforward to check that the sequence \( \{x^k\} \) generated by DCA (via \eqref{eq:seq_eg01}) from any initial point \( x^0 > 0 \) is well-defined, decreasing, and converges to the minimizer \( x^\star = 0 \).
\end{remark}

Example \ref{eg:01} and Remark \ref{rmk:1} indicate that we need more refined assumptions on the DC decomposition to guarantee the well-definedness of DCA. A straightforward answer is: \emph{DCA is well-defined, if and only if, for all $k\in \N$, $x^k\in \dom \partial h$ and \eqref{prob:Pk} has an optimal solution}. Next, we propose an additional assumption (\AssumpB) to guarantee the well-definedness of DCA.\\
\noindent
\fbox{
	\parbox{\dimexpr\textwidth-2\fboxsep-2\fboxrule\relax}{
		\noindent (\AssumpB)\label{assump:B}
		\begin{itemize}
			\item $\emptyset \neq \dom \partial (g+\chi_{\setC}) \subset \dom \partial h$.
			\item All subproblems \eqref{prob:Pk} have optimal solutions (may not be unique).
		\end{itemize}
	}
}
\begin{proposition}
	
	Under \textbf{Assumptions A and B}, the sequence $\{x^k\}$ generated by DCA for \eqref{prob:P} with any initial point $x^0\in \dom \partial h$ is well-defined.
	
\end{proposition}
\begin{proof}
	Given $x^k\in \dom \partial h\neq \emptyset$, the non-emptiness of the solution set of \eqref{prob:Pk} implies that there exists $x^{k+1}$ such that 
	$$y^k\in \partial (g+\chi_{\setC})(x^{k+1}).$$
	Hence $x^{k+1}\in \dom \partial (g+\chi_{\setC}).$
	Then by the inclusion $\dom \partial (g+\chi_{\setC}) \subset \dom \partial h$, we get
	$$x^{k+1}\in \dom \partial h.$$
	By induction, given $x^0\in \dom \partial h$, we deduce that $$x^k\in \dom \partial h\neq \emptyset, \forall k\in \N.$$
	Hence, the sequence $\{x^k\}$ generated by DCA is well-defined. 
\end{proof}

\paragraph*{Discussion}
\begin{itemize}
	\item In the DC decomposition \eqref{eq:dcd01_eg01}, \AssumpB\ is violated because 
	$\dom\partial g=[0,\infty)\not\subset\dom\partial h=(0,\infty)$. 
	By contrast, the decomposition \eqref{eq:dcd02_eg01} satisfies \AssumpB. 
	Consequently, DCA based on \eqref{eq:dcd01_eg01} is \emph{not} well-defined, whereas DCA based on \eqref{eq:dcd02_eg01} \emph{is} well-defined.
	
	\item \AssumpB\ should be viewed as a pre-KL condition: it guarantees that the algorithmic objects needed by the later convergence analysis are well-defined before any regularity argument is applied.
	
	\item Under \textbf{Assumptions A and B}, boundedness of the subgradient sequence $\{y^k\}$ is \emph{not} guaranteed in general. The decomposition \eqref{eq:dcd02_eg01} provides an illustrative case where $x^k\to 0$ (hence $\{x^k\}$ is bounded) but $y^k\to-\infty$ (thus $\{y^k\}$ is unbounded).
	
	\item \AssumpA\ is \emph{not} necessary for the well-definedness of DCA. 
	Consider $g(x)=x^2/2$ and $h(x)=e^x$. 
	The problem \eqref{prob:P} has no solution, so \AssumpA\ fails. 
	Nevertheless, DCA is still well-defined for any $x^0\in\R$: since $y^k=e^{x^k}$, the subproblem $\min_x\{g(x)-\langle y^k,x\rangle\}$ has a unique minimizer $x^{k+1}=e^{x^k}$. 
	Hence $x^k\to\infty$ and $f(x^k)=\tfrac12(x^k)^2-e^{x^k}\to-\infty$.
	
	\item The non-emptiness of the solution set of each convex subproblem \eqref{prob:Pk} in \AssumpB\ is automatically satisfied in many practical situations. 
	In particular:
	\begin{itemize}
		\item if $g$ is strongly convex on $\mathcal C$, then $x\mapsto g(x)-\langle y^k,x\rangle$ is strongly convex for every $y^k$, hence \eqref{prob:Pk} has a unique minimizer;
		\item if $\mathcal C$ is nonempty, compact, and convex, then the convex function $x\mapsto g(x)-\langle y^k,x\rangle$ attains its minimum over $\mathcal C$;
		\item more generally, for a fixed $y^k$, if $x\mapsto g(x)-\langle y^k,x\rangle$ is coercive or level-bounded on $\mathcal C$, then \eqref{prob:Pk} admits a minimizer.
	\end{itemize}	
	Regarding the existence of solutions to \eqref{prob:P}, a convenient sufficient condition is that $f$ is proper, closed, and coercive (or level-bounded) on $\mathcal C$. This ensures that $\min_{\mathcal C} f$ attains its infimum. These conditions are not required for DCA to be well-defined; they are provided as practical safeguards when the existence of global minimizers is desired.
\end{itemize}

\begin{remark}[Two common pitfalls]
	\leavevmode
	\begin{itemize}
		\item \textbf{“$f$ being lower bounded $\Rightarrow$ \eqref{prob:P} has a solution.”} False in general.  
		Counterexample: $g(x)=e^x$, $h(x)=0$. Then $f(x)=e^x$ is bounded below by $0$ on $\R$, but the minimum is not attained; consequently \eqref{prob:Pk} (here, $\min\{e^x:x\in\R\}$) has no solution and DCA is not well-defined.
		
		\item \textbf{“Level-bounded $f$ $\Rightarrow$ \eqref{prob:P} has a solution / $f$ is lower bounded / $\{y^k\}$ is bounded.”} All are false without extra hypotheses.  
		Counterexample 1: $f(x)=\ln x$ for $x\in(0,\infty)$ and $f(x)=+\infty$ otherwise. Then $f$ is level-bounded, yet $\min f$ is not attained and $f$ is not lower bounded on $\R$ (also $f$ is not lower semicontinuous as an extended-real-valued function).  
		Counterexample 2: in Example~\ref{eg:01} with the DC representation \eqref{eq:dcd02_eg01}, $f$ is level-bounded while $\{y^k\}$ is unbounded even though $\{x^k\}$ is bounded.
	\end{itemize}
\end{remark}

\section{Classical Convergence Properties of DCA}\label{sec:classicalcvofdca}

In this section, we revisit several well-known lemmas (see, e.g., \cite{Pham1997}) that highlight the classical convergence properties of DCA. Although these results are known, they were gradually discovered by different authors over time under various assumptions, making the conclusions scattered and fragmented. Here we provide clean and concise proofs for each under our proper assumptions, as they will serve as foundational templates for establishing the convergence of many DCA variants. This section is intentionally preparatory rather than novelty-driven: its purpose is to isolate the exact DCA identities that will later feed into the unified framework and its variants. In the sequel, \textbf{we assume that the sequence $\{x^k\}$ generated by DCA, starting from an initial point $x^0 \in \dom \partial h$ for \eqref{prob:P} under \AssumpA, is well-defined}.

\subsection[Classical Convergence Results]{Classical Convergence Results of $\{f(x^k)\}$ and $\{\|x^{k+1}-x^{k}\|\}$}
\begin{lemma}[Non-Increasing and Convergence of $\{f(x^k)\}$]\label{lemma:nonincreasingoffxk}
	The sequence $\{f(x^k)\}$ is non-increasing and convergent.
\end{lemma}
\begin{proof}
	For every $k=0,1,\ldots$, DCA generates the sequence $\{x^k\}$ as
	\begin{equation}
		\label{eq:update}
		x^{k+1}\in \argmin\{g(x) - \langle y^k , x \rangle : x\in \setC\},
	\end{equation}
	where $y^k\in \partial h(x^k)$. By the convexity of $h$ and $y^k\in \partial h(x^k)$, we have
	\begin{equation}
		\label{eq:convexityofh}
		h(x^{k+1})\geq h(x^{k}) + \langle x^{k+1}-x^{k}, y^k \rangle.
	\end{equation}
	Then, for every $k\in \N$, we have 
	\begin{eqnarray*}
		f(x^{k+1})&=&g(x^{k+1})-h(x^{k+1})\\
		&\overset{\eqref{eq:convexityofh}}{\leq}&g(x^{k+1})-\left( h(x^k) + \langle x^{k+1} - x^k,y^k \rangle \right)\\
		&\overset{\eqref{eq:update}}{=}& \min\{ g(x) - \langle x,y^k \rangle  : x\in \setC \} -  h(x^k) + \langle x^k,y^k \rangle\\
		&\leq& g(x^k) - h(x^k) = f(x^k),
	\end{eqnarray*}
	which implies that $\{f(x^k)\}$ is non-increasing. \\
	The non-emptiness of the solution set of \eqref{prob:P} suggests that $f$ is lower bounded over $\setC$. Consequently, the convergence of $\{f(x^k)\}$ follows from its non-increasing and the lower boundedness.
\end{proof}

Now, consider the case where $g$ is $\rho_g$-convex\footnote{A function $g\in \Gamma_0(\R^n)$ is called $\rho_g$-convex over a convex set $\setC\subset \R^n$, if and only if, for all $y\in \partial g(x)\neq \emptyset$, $x\in \setC$ and $z\in \setC$, we have $g(z)\geq g(x) + \langle z-x, y \rangle + \frac{\rho_g}{2}\|z-x\|^2$.} over $\setC$ (with a modulus $\rho_g\geq 0$) and $h$ is $\rho_h$-convex over $\setC$ (with a modulus $\rho_h\geq 0$), we have the next lemma:
\begin{lemma}\label{lemma:suffdescent&squaresummable}
	Assuming $\rho_g+\rho_h>0$ over $\mathcal{C}$, the following properties hold: 
	\begin{itemize}
		\item(\textbf{Sufficient Descent Property}) \begin{equation}
			\label{eq:suffdecprop}
			f(x^k) - f(x^{k+1}) \geq \frac{\rho_g+\rho_h}{2}\|x^k-x^{k+1}\|^2, \quad\forall k=1,2,\ldots.
		\end{equation}
		\item (\textbf{Square Summable Property}) $$\sum_{k\geq 0} \|x^{k+1}-x^{k}\|^2<\infty,$$
		hence $\|x^{k+1}-x^{k}\|\to 0$ as $k\to \infty$.
	\end{itemize}
\end{lemma}
\begin{proof}
	\textbf{(Sufficient Descent Property)}: For every $k=0,1,\ldots$, we have the following inequalities:\\
	$\rhd$ By the $\rho_h$-convexity of $h$ and $y^k\in \partial h(x^k)$, we have
	\begin{equation}
		\label{eq:strongconvexityofh}
		\boxed{h(x^{k+1})\geq h(x^{k}) + \langle x^{k+1}-x^{k}, y^k \rangle + \frac{\rho_h}{2}\|x^{k+1}-x^{k}\|^2.}
	\end{equation}
	$\rhd$ By the $\rho_g$-convexity of $g$, we have 
	\begin{equation}
		\label{eq:strongconvexityofg}
		\boxed{g(x^{k})\geq g(x^{k+1}) + \langle x^{k}-x^{k+1}, w \rangle + \frac{\rho_g}{2}\|x^{k+1}-x^{k}\|^2, \quad \forall w\in \partial g(x^{k+1}).}
	\end{equation}
	$\rhd$ By the first order optimality condition for the convex problem
	$$x^{k+1}\in \argmin\{g(x) - \langle y^k , x \rangle : x\in \setC\}$$
	where $y^k\in \partial h(x^k)$, we have
	$$0\in \partial g(x^{k+1}) + N_{\setC}(x^{k+1}) - y^k,$$
	where $N_{\setC}(x^{k+1})$ denotes the normal cone of $\setC$ at $x^{k+1}$.
	Thus $$y^k \in\partial h(x^k) \cap (\partial g(x^{k+1}) + N_{\setC}(x^{k+1})).$$ Then, by $y^k \in\partial g(x^{k+1}) + N_{\setC}(x^{k+1})$, we get that there exist $v\in N_{\setC}(x^{k+1})$ and $w\in \partial g(x^{k+1})$ such that $y^k = w+v$. Substituting $w$ with $y^k-v$ in \eqref{eq:strongconvexityofg}, we obtain 
	\begin{equation}\label{eq:strongconvexityofg-bis}
		\boxed{g(x^{k}) + \langle x^{k+1}-x^{k}, y^k-v \rangle - \frac{\rho_g}{2}\|x^{k+1}-x^{k}\|^2 \geq g(x^{k+1}).}
	\end{equation}
	$\rhd$ It follows from $v\in N_{\setC}(x^{k+1})$ and $x^k\in \setC$ that  
	\begin{equation}
		\label{eq:normconeineq}
		\boxed{\langle x^k-x^{k+1}, v \rangle \leq 0.}
	\end{equation}
	Now, we can deduce that
	\begin{eqnarray*}
		f(x^{k+1})&=&g(x^{k+1})-h(x^{k+1})\\
		&\overset{\eqref{eq:strongconvexityofh}}{\leq}&g(x^{k+1})-\left( h(x^k) + \langle x^{k+1} - x^k,y^k \rangle + \frac{\rho_h}{2}\|x^{k+1}-x^{k}\|^2 \right)\\
		&\overset{\eqref{eq:strongconvexityofg-bis}}{\leq}& g(x^k) - h(x^k) + \langle x^{k}-x^{k+1}, v \rangle - \frac{\rho_g+\rho_h}{2}\|x^{k+1}-x^{k}\|^2\\
		&\overset{\eqref{eq:normconeineq}}{\leq}& f(x^k) - \frac{\rho_g+\rho_h}{2}\|x^{k+1}-x^{k}\|^2, \forall k=1,2,\ldots.
	\end{eqnarray*}
	\textbf{(Square Summable Property)}: Summing \eqref{eq:suffdecprop} for $k$ from $1$ to $N$ ($\geq 1$), we have 
	\begin{align*}
		\sum_{k=1}^{N} \frac{\rho_g+\rho_h}{2}\|x^k-x^{k+1}\|^2 &\leq \sum_{k=1}^{N}  f(x^k) - f(x^{k+1}) \\
		&= f(x^1)-f(x^{N+1}) \\
		&\leq f(x^1) - \min\{f(x):x\in \setC\},
	\end{align*}
	where the last inequality holds for any $N\geq 1$. \\
	Taking $N\to \infty$, and by the lower boundedness of $f$ over $\setC$, then
	$$\sum_{k\geq 1} \frac{\rho_g+\rho_h}{2}\|x^k-x^{k+1}\|^2 < \infty.$$
	It follows immediately by $\rho_g+\rho_h>0$ and the finiteness of $\|x^0-x^1\|$ that
	$$\sum_{k\geq 0}\|x^k-x^{k+1}\|^2 < \infty,$$
	and as a consequence, $\|x^k-x^{k+1}\|\to 0$ as $k\to \infty$.
\end{proof}

\paragraph*{Discussion} 
\begin{itemize}
	\item The result $\|x^k-x^{k+1}\|\to 0$ can also be derived from the sufficient descent property and the convergence of $\{f(x^k)\}$ as
	$$0\leq \frac{\rho_g+\rho_h}{2} \|x^k-x^{k+1}\|^2 \leq f(x^k) - f(x^{k+1}) \xrightarrow{k\to \infty} 0.$$
	
	\item Without the assumption $\rho_g + \rho_h>0$, the sequence $\{\|x^k - x^{k+1}\|\}$ may not converge to $0$. Consider the following example:
	\begin{example}\label{eg:02}
		Let $g(x) = \sup\{-x,0,x-1\}$, $h(x) = \sup\{-x,0\}$, and $C= \R$. The functions $g$ and $h$ are piecewise linear and convex (polyhedral convex), but neither strongly convex nor strictly convex. Starting DCA from an initial point $x^0 \in (0,1)$, we have $\partial h(x^0) = \{0\}$, and $x^1 \in \argmin \{g(x)\} = [0,1]$. Hence, DCA could generate a sequence $\{x^k\}$ as $\{0.1, 0.9, 0.1, 0.9, \cdots\}$. Thus, $\{\|x^k-x^{k+1}\|\}$ is a constant sequence $\{0.8, 0.8, \cdots\}$ whose limit is nonzero, and the square summable property is not satisfied. Note that the sequence $\{f(x^k)\}$ is the constant zero-sequence, which is convergent but does not satisfy the sufficient descent property. 
	\end{example}
	
	\item In practice, the condition $\rho_g+\rho_h>0$ does not impose a significant restriction. This is because for any DC function $f:x\mapsto g(x) - h(x)$, where \AssumpA~holds, we can introduce a $\rho$-convex ($\rho>0$) function $\varphi \in \Gamma_0(\R^n)$ such that $\dom \varphi \supset \dom h$. This allows us to derive a new DC decomposition $(g+\varphi)(x) - (h+\varphi)(x)$, where both $g+\varphi$ and $h+\varphi$ are $\rho$-convex $\Gamma_0(\R^n)$ functions, i.e., $\rho_{g+\varphi} + \rho_{h+\varphi} = 2\rho > 0$ is satisfied. A commonly used example is to take $\varphi(x) = \frac{\rho}{2}\|x\|^2$ for any $\rho>0$, then the DCA associated with this DC decomposition is expressed as:
	\[
	y^k \in \partial h(x^k),\qquad 
	x^{k+1}=\argmin_x\Big\{g(x)+\tfrac{\rho}{2}\|x\|^2-\langle \rho x^k+y^k,\,x\rangle\Big\}.
	\]
	Completing the square gives
	\[
	x^{k+1}=\prox_{\,g/\rho}\!\Big(x^k+\tfrac{y^k}{\rho}\Big),
	\]
	which is essentially the proximal point method applied to the DC function $g - h$. Here the proximal operator of $g\in\Gamma_0(\mathbb{R}^n)$ is
    $$\prox_g(z):= \underset{x}{\argmin}\{ g(x) + \frac{1}{2}\|x-z\|^2\}.$$
	
	\item In some variants of DCA, it is not the sequence $\{f(x^k)\}$ that satisfies the sufficient descent property, but rather the sequence of an auxiliary function (e.g., Lyapunov function or energy function). For example, the proximal DCA with extrapolation (pDCAe) proposed in \cite{wen2018proximal} is a variant of DCA that introduces Nesterov's extrapolation into DCA using a proximal DC decomposition. The sequence $\{f(x^k)\}$ is non-monotone, but we can prove in a manner similar to Lemma \ref{lemma:suffdescent&squaresummable} that the sequence $\{\zeta^k := f(x^k) + \frac{L}{2}\|x^k-x^{k-1}\|^2\}$ is monotone and enjoys the sufficient descent property, which further guarantees the convergence of the non-monotone sequence $\{f(x^k)\}$ and the square summable property.
\end{itemize}

\subsection{Classical Convergence Rates of DCA}
\begin{corollary}[Classical Convergence Rates of DCA]\label{cor:conv_rate_1/N}
	Let $f^\star=\lim_{k\to \infty} f(x^k)$. The following convergence rates hold: 
	\begin{itemize}
		\item \textbf{(Convergence Rate for $\{f(x^k) - f(x^{k+1})\}$)}
		$$\frac{1}{N+1}\sum_{k=0}^{N} f(x^k) - f(x^{k+1}) \leq \frac{f(x^0) - f^\star}{N+1} = O\left(\frac{1}{N}\right),$$
		$$\min_{0\leq k\leq N} \{f(x^k) - f(x^{k+1})\} \leq  O\left(\frac{1}{N}\right).$$
		\item \textbf{(Convergence Rate for $\{\|x^{k+1} - x^{k}\|\}$)} If $\rho_g+\rho_h>0$ over $\setC$, then
		$$\frac{1}{N+1}\sum_{k=0}^{N} \|x^{k+1}-x^{k}\|^2 \leq \frac{2(f(x^0) - f^\star)}{(\rho_g+\rho_h)(N+1)} = O\left(\frac{1}{N}\right),$$
		$$\min_{0\leq k\leq N} \{\|x^{k+1}-x^{k}\|\} \leq O\left(\frac{1}{\sqrt{N}}\right).$$
	\end{itemize}
\end{corollary}
\begin{proof}
	\textbf{(Convergence Rate for $\{f(x^k) - f(x^{k+1})\}$)}: By the non-increasing and convergence of the sequence $\{f(x^k)\}$ (Lemma \ref{lemma:nonincreasingoffxk}), we have 
	$$\sum_{k=0}^N f(x^k) - f(x^{k+1}) \leq \sum_{k\geq 0} f(x^k) - f(x^{k+1}) = f(x^0) - f^\star.$$
	Hence $$\frac{1}{N+1}\sum_{k=0}^{N} f(x^k) - f(x^{k+1}) \leq \frac{f(x^0) - f^\star}{N+1} = O\left(\frac{1}{N}\right).$$
	It follows that 
	$$\min_{0\leq k\leq N} \{f(x^k) - f(x^{k+1})\} \leq \frac{1}{N+1}\sum_{k=0}^{N} f(x^k) - f(x^{k+1}) \leq  O\left(\frac{1}{N}\right).$$
	\textbf{(Convergence Rate for $\{\|x^{k+1} - x^{k}\|\}$)}: We get from the proof of Lemma \ref{lemma:suffdescent&squaresummable} that
	$$\sum_{k=0}^{N} \frac{\rho_g+\rho_h}{2}\|x^k-x^{k+1}\|^2 \leq 
	f(x^0) - \min\{f(x):x\in \setC\} = f(x^0) - f^\star.
	$$
	Hence,
	$$\sum_{k=0}^{N} \|x^k-x^{k+1}\|^2 \leq 
	\frac{2(f(x^0) - f^\star)}{\rho_g+\rho_h}.$$
	It follows that
	$$\frac{1}{N+1}\sum_{k=0}^{N} \|x^{k+1}-x^{k}\|^2 \leq \frac{2(f(x^0) - f^\star)}{(\rho_g+\rho_h)(N+1)} = O\left(\frac{1}{N}\right).$$
	Hence, \begin{align*}
		\min_{0\leq k\leq N} \{\|x^k-x^{k+1}\|\} &\leq \sqrt{\frac{1}{N+1}\sum_{k=0}^{N} \|x^{k+1}-x^{k}\|^2}\\
		&\leq 
		\sqrt{\frac{2(f(x^0) - f^\star)}{(\rho_g+\rho_h)(N+1)}}
		= O\left(\frac{1}{\sqrt{N}}\right).
	\end{align*}
\end{proof}
\paragraph*{Discussion}
\begin{itemize}
	\item Corollary \ref{cor:conv_rate_1/N} indicates that the convergence rate of DCA depends on $f(x^0) - f^\star$ (i.e., the quality of the initial point) for the sequence $\{f(x^k) - f(x^{k+1})\}$, and on $\rho_g + \rho_h$ (i.e., the quality of the DC decomposition) for the sequence $\{\|x^{k+1} - x^k\|\}$. A smaller $f(x^0) - f^\star$ and a larger $\rho_g + \rho_h$ lead to faster convergence. The convergence rate is generally sublinear, though it is possible to derive a linear convergence rate under stronger assumptions. We will explore improved convergence rates under the Kurdyka--{\L}ojasiewicz property in \S\ref{sec:convergencerate}.
	\item Regarding Corollary \ref{cor:conv_rate_1/N}, using $\|x^{k+1} - x^k\| \leq \varepsilon$ or $|f(x^{k+1}) - f(x^k)| \leq \varepsilon$ for some given tolerance $\varepsilon > 0$ as stopping criteria for DCA (a commonly used approach) is extremely risky, as there is no guarantee of the optimality of the computed solution. For example, if we choose any initial point $x^0$ and select a DC decomposition $g-h$ and $\varepsilon$ such that
	$$\rho_g + \rho_h > \frac{f(x^0) - f^\star}{\varepsilon^2},$$
	then Corollary \ref{cor:conv_rate_1/N} implies that DCA will terminate after just one iteration, without any assurance of the optimality of $x^0$! Therefore, it is crucial to develop a rigorous stopping criterion that ensures the optimality of the computed solution; however, this is beyond the scope of the present manuscript. Throughout this paper, we will consider DCA without a stopping condition, allowing the generation of infinite sequences, and study their convergence properties.
\end{itemize}

\subsection{DC Optimality Conditions}
\subsubsection{DC Criticality}

We recall the classical optimality condition for the DC program, known as \emph{DC criticality}, which is defined as follows:

\begin{definition}[DC Critical Point]
	A point \(x^\star\) is called a \emph{DC critical point} of the DC program \eqref{prob:P} if and only if
	\begin{equation}
		\label{DC_Criticality}
		\partial (g + \chi_{\setC})(x^\star) \cap \partial h(x^\star) \neq \emptyset.
	\end{equation}
	In particular:
	\begin{itemize}
		\item If \(\emptyset \neq \partial h(x^\star) \subset \partial (g + \chi_{\setC})(x^\star)\), then \(x^\star\) is called a \emph{strongly DC critical point}.
		\item If \(\setC = \R^n\), then the DC criticality condition \eqref{DC_Criticality} simplifies to 
		\[
		\partial g(x^\star) \cap \partial h(x^\star) \neq \emptyset,
		\]
		or equivalently,
		\[
		0 \in \partial g(x^\star) - \partial h(x^\star).
		\]
		\item If \(\ri(\dom g) \cap \ri(\setC) \neq \emptyset\), then the DC criticality condition \eqref{DC_Criticality} reduces to:
		\[
		\left(\partial g(x^\star) + N_{\setC}(x^\star)\right) \cap \partial h(x^\star) \neq \emptyset,
		\]
		where \(N_\setC(x^\star)\) is the normal cone of \(\setC\) at \(x^\star\), defined as
		\[
		N_\setC(x^\star) := \{y \in \R^n : \langle y, x - x^\star \rangle \leq 0, \ \forall x \in \setC \}.
		\]
	\end{itemize}
\end{definition}

\begin{theorem}[Subsequential Convergence of $\{x^k\}$ to DC Critical Points]\label{thm:subseqconv}
	Let $\{x^k\}$ and $\{y^k\} \subset \{\partial h(x^k)\}$ be two well-defined and bounded sequences generated by DCA starting from an initial point $x^0 \in \dom \partial h$ for the DC program \eqref{prob:P} under \AssumpA. If $\rho_g + \rho_h > 0$ over $\setC$, then any cluster point of the sequence $\{x^k\}$ is a DC critical point of \eqref{prob:P}.
\end{theorem}
\begin{proof}
	By the boundedness of $\{x^k\}$, there exists a convergent subsequence of $\{x^k\}$ (by the Bolzano-Weierstrass theorem), denoted $\{x^{k_j}\}_j$, converging to a limit point $x^\star$. If either $g$ or $h$ is strongly convex over $\setC$, then we have $\|x^k - x^{k+1}\| \to 0$ by Lemma \ref{lemma:suffdescent&squaresummable}, which implies 
	$$x^{k_j} \to x^\star \text{ and } x^{k_j+1} \to x^\star.$$
	The first-order optimality condition for \eqref{prob:Pk} at $k = k_j$ is
	\begin{equation}\label{eq:subproboptcond}
		y^{k_j} \in \left(\partial g(x^{k_j+1}) + N_{\setC}(x^{k_j+1})\right) \cap \partial h(x^{k_j}).    
	\end{equation}
	By the boundedness of the sequence $\{y^{k_j}\}$, the set of cluster points of $\{y^{k_j}\}$ is non-empty. Without loss of generality, suppose that the sequence $\{y^{k_j}\}$ is convergent. Taking the limit in \eqref{eq:subproboptcond}, and using the closedness of the graphs of $\partial g$, $N_{\setC}$, and $\partial h$, we conclude that the limit point of $\{y^{k_j}\}$ is included in 
	$\left(\partial g(x^{*}) + N_{\setC}(x^{*})\right) \cap \partial h(x^{*})$. Thus, $\left(\partial g(x^{*}) + N_{\setC}(x^{*})\right) \cap \partial h(x^{*}) \neq \emptyset$, i.e., $x^\star$ is a DC critical point. 
\end{proof}

\begin{remark}
	The boundedness of the sequences $\{x^k\}$ and $\{y^k\}$ is crucial in the proof of Theorem \ref{thm:subseqconv}. However, note that \AssumpA\ and \AssumpB\ do not guarantee their boundedness. Example \ref{eg:01}, with the DC representation \eqref{eq:dcd02_eg01} (satisfying \AssumpA\ and \AssumpB), demonstrates that $x^k \to 0$ and $y^k \to -\infty$.
\end{remark}

The next theorem shows two important properties of a local minimizer of \eqref{prob:P}, where the first assertion is often overlooked.

\begin{theorem}
	\label{thm:DCoptcond}
	Let $x^\star$ be a local minimizer of \eqref{prob:P}. Then:
	\begin{itemize}
		\item $x^\star$ may not be a DC critical point, i.e., $\partial (g+\chi_{\setC})(x^\star) \cap \partial h(x^\star) \neq \emptyset$ may not hold.
		\item If both $g+\chi_{\setC}$ and $h$ are subdifferentiable at $x^\star$, then $x^\star$ is a strongly DC critical point, i.e., $\emptyset \neq \partial h(x^\star) \subset \partial (g+\chi_{\setC})(x^\star)$, which often coincides with d(irectional)-stationarity under certain technical assumptions (see e.g., $\ri(\dom g \cap \setC) \cap \ri(\dom h) \neq \emptyset$; \cite[Theorem 1]{DCA30}).
	\end{itemize}
\end{theorem}

\begin{proof}
	Example \ref{eg:01} demonstrates that the local minimizer $x^\star=0$ is not a DC critical point since $\partial h(0) = \emptyset$. For proving the second assertion, there exists a neighborhood $\mathcal{U}$ around a local minimizer $x^\star$ such that $$g(x^\star) +\chi_{\setC}(x^\star) - h(x^\star)\leq g(x) +\chi_{\setC}(x) - h(x), \forall x\in \setU.$$
	As $x^\star$ is a local minimizer, we have $x^\star\in \dom g \cap \setC$, implying that $\chi_{\setC}(x^\star)=0$. Hence,
	\begin{equation}
		\label{eq:thm:DCoptcond_eq1}
		g(x^\star)\leq g(x) +\chi_{\setC}(x)  + (h(x^\star) - h(x)), \forall x\in \setU.
	\end{equation}
	By the convexity of $h$ over $\setU$, we get for all $x\in \setU$ that 
	\begin{equation}
		\label{eq:thm:DCoptcond_eq2}
		h(x^\star) - h(x)\leq - \langle x - x^\star, y^*  \rangle, \forall y^*\in \partial h(x^\star).
	\end{equation}
	Combining \eqref{eq:thm:DCoptcond_eq1} and \eqref{eq:thm:DCoptcond_eq2}, we have 
	\begin{equation}
		\label{eq:thm:DCoptcond_eq3}
		g(x^\star)\leq g(x) +\chi_{\setC}(x)  -  \langle x - x^\star, y^*  \rangle, \forall x\in \setU.
	\end{equation}
	Hence, $x^\star$ is a local minimizer of $x\mapsto g(x) +\chi_{\setC}(x)  -  \langle x - x^\star, y^*  \rangle$ over $\setU$. The convexity of $g$ and $\chi_{\setC}$ implies that $x^\star$ is a global minimizer of the convex function $x\mapsto g(x) +\chi_{\setC}(x)  -  \langle x - x^\star, y^*  \rangle$ over $\R^n$, indicating that 
	$$y^*\in \partial (g + \chi_{\setC})(x^\star).$$
	This is true for all $y^*\in \partial h(x^\star)$. Hence $\partial h(x^\star)\subset \partial (g+\chi_{\setC})(x^\star)$. Combining the assumption that both $g + \chi_{\setC}$ and $h$ are subdifferentiable at $x^\star$, we get  
	$$\emptyset \neq \partial h(x^\star)\subset \partial (g+\chi_{\setC})(x^\star).$$
\end{proof}

\subsubsection{Weaknesses of DC Criticality}

The DC criticality condition is generally regarded as a weak optimality condition for DC programs. As previously observed in Example \ref{eg:01}, \textbf{a local minimizer of a DC program may not always be a DC critical point}, which is a fact often overlooked in the literature. Furthermore, even when a local minimizer is a DC critical point, Cui-Pang-Sen noted in \cite{cui2018composite} that:
\[
\begin{aligned}
	\text{Local minimum} &\subset \text{D(irectional) stationarity} \subset \text{L(imiting) stationarity} \\
	&\subset \text{C(larke) stationarity} \subset \text{DC criticality},
\end{aligned}
\]
which underscores the inherent weakness of the DC criticality. The following example demonstrates that a DC critical point may not necessarily be a local minimizer.

\begin{example}\label{eg:weaknessofdccriticality}
	Consider the functions 
	$$g(x) = \max\{0, x\} + \chi_{\{x \geq -1\}}(x), \quad h(x) = \max\{0, -x\}.$$
	Clearly, both \AssumpA\ and \AssumpB\ are satisfied. The point $x^\star = 0$ is a DC critical point because $\partial g(0) = [0,1]$, $\partial h(0) = [-1,0]$, and 
	$0 \in \partial g(0) - \partial h(0) = [0,2].$
	Starting DCA from the initial point $x^0 = 0$ can generate the zero-sequence by taking $y^k = 0 \in \partial h(0)$ and $x^{k+1} = 0 \in \argmin\{g(x)\} = [-1,0]$ for all $k \in \N$. However, $0$ is not a local minimizer of $\min\{g(x) - h(x)\}$.
\end{example}

It is worth noting the significant impact of a ``good DC decomposition'' and the ``choice of $y^k \in \partial h(x^k)$'' on the quality of the DC critical point.

\begin{itemize}
	\item \textbf{Influence of a Good DC Decomposition:} In Example \ref{eg:weaknessofdccriticality}, consider a different DC decomposition:
	$$g(x) = x + \chi_{\{x \geq -1\}}(x), \quad h(x) = 0.$$
	In this case, $\partial g(0) = \{1\}$ and $\partial h(0) = \{0\}$. Hence, $0$ is no longer a DC critical point. However, a DC critical point can be found at $-1$, where $\partial g(-1) = (-\infty,1]$ and $\partial h(-1) = \{0\}$. Starting DCA from any initial point $x^0 \in \dom \partial h = \R$, we obtain $x^k \in \argmin\{g(x)\} = \{-1\}$ for all $k \geq 1$. Therefore, DCA generates the sequence $\{x^0, -1, -1, \ldots\}$, which converges to $-1$, the optimal solution.
	
	\item \textbf{Influence of the Choice of $y^k \in \partial h(x^k)$:} In the original DC decomposition from Example \ref{eg:weaknessofdccriticality}, we have $\partial h(0) = [-1,0]$. If we choose $y^k \in [-1,0)$ (i.e., with $y^k \neq 0$), then $x^{k+1} \in \argmin\{g(x) - y^k x\} = \{-1\}$, and hence the sequence $\{x^k\}$ converges to the optimal solution $-1$.
\end{itemize}

\section{A Unified Convergence Framework for DCA}\label{sec:cvofdcawithL}

The iterate sequence $\{x^k\}$ generated by DCA need not converge even when the objective values $\{f(x^k)\}$ do (see Example~\ref{eg:02}). Thus, additional structure is required to guarantee convergence of $\{x^k\}$. Unlike abstract descent methods, DCA requires two extra pieces of bookkeeping before KL regularity can be used: the algorithm must be well-defined, and the residual estimate has to be extracted from the convex subproblems together with the chosen subgradients. In nonsmooth, nonconvex optimization, two widely used regularity assumptions are the \emph{{\L}ojasiewicz subgradient inequality} and the \emph{Kurdyka--{\L}ojasiewicz (KL) property}. These conditions impose mild geometric regularity that often yields global convergence and explicit rates for algorithmic iterates. Convergence of DCA under the {\L}ojasiewicz subgradient inequality was analyzed by Le Thi--Huynh--Pham~\cite{le2018convergence}. Building on this insight, we develop a unified framework that covers both the {\L}ojasiewicz subgradient inequality and the broader KL property, and use it to derive convergence and rate results for DCA.

\subsection[Key Assumptions for DCA Convergence]{Key Assumptions for the Convergence of $\{x^k\}$}
The following theorem provides key sufficient conditions for establishing the convergence of \(\{x^k\}\) (including DCA and its variants).

\begin{theorem}\label{thm:generalcvx}
	Let $\{x^k\}$ be a well-defined and bounded sequence with $x^k \neq x^{k+1}$ for all $k \in \N$. Suppose that:
	\begin{itemize}
		\item[] \Hone~\textbf{(Lyapunov Assumption)} There exists a (Lyapunov) function $\Psi:\R^n \to [-\infty, \infty]$ such that the sequence $\{\Psi(x^k)\}$ is well-defined, non-negative, and converging to 0.
		
		\item[] \Htwo~\textbf{(Sufficient Descent Assumption)} For sufficiently large $k$, there exists $D > 0$ such that 
		\begin{equation}
			\label{eq:H2}
			\Psi(x^k) - \Psi(x^{k+1}) \geq D \|x^k - x^{k+1}\|^2.
		\end{equation}
		
		\item[] \Hthree~\textbf{(Regularity Assumption)} For sufficiently large $k$, there exist 
		\begin{itemize}
			\item a differentiable concave function $\varphi:[0,\eta) \to \R_+$ (for some $\eta > \Psi(x^k)$, $\eta \leq \infty$) with $\varphi(0) = 0$ and $\varphi' > 0$ over $(0, \eta)$,
			\item two non-negative constants $C_1$ and $C_2$ with $C_1 + C_2 > 0$,
		\end{itemize}
		such that 
		\begin{equation}
			\label{eq:H3}
			\varphi'(\Psi(x^k)) \times (C_1\|x^k - x^{k+1}\| + C_2\|x^{k-1} - x^k\|) \geq 1.
		\end{equation}
	\end{itemize}
	Then, for sufficiently large $k$, we have the inequality
	\begin{equation}\label{eq:residualrelation}
		\frac{3}{4}\|x^k - x^{k+1}\| \leq  \frac{1}{4}\|x^{k-1} - x^k\| + \frac{\max\{C_1, C_2\}}{D} \left(\varphi(\Psi(x^k)) - \varphi(\Psi(x^{k+1}))\right),
	\end{equation}
	and the sequence $\{x^k\}$ is convergent. 
\end{theorem}

\begin{proof}
	For large enough $k$, by the concavity of $\varphi$, we get 
	\begin{eqnarray*}
		\varphi(\Psi(x^{k})) - \varphi(\Psi(x^{k+1}))
		&\geq& \varphi'(\Psi(x^{k})) (\Psi(x^{k}) - \Psi(x^{k+1})) \\
		&\overset{\eqref{eq:H2}}{\geq}&D\|x^{k} - x^{k+1}\|^2 \varphi'(\Psi(x^{k}))\\
		&\overset{\eqref{eq:H3}}{\geq}&\frac{D\|x^{k} - x^{k+1}\|^2}{C_1\|x^{k} - x^{k+1}\| + C_2\|x^{k-1}-x^k\|} \\
		&\geq&\frac{D}{\max\{C_1,C_2\}}\frac{\|x^{k} - x^{k+1}\|^2 }{\|x^{k} - x^{k+1}\| + \|x^{k-1}-x^k\|}. 
	\end{eqnarray*}
	
	It follows from Young's inequality ($a\leq a^2/b + b/4$ with $a,b>0$) that
	
	$$
	\frac{3}{4}\|x^{k} - x^{k+1}\| - \frac{1}{4} \|x^{k-1}-x^k\| \leq \frac{\|x^{k} - x^{k+1}\|^2}{\|x^{k} - x^{k+1}\| + \|x^{k-1}-x^k\|}.
	$$
	Hence, for large enough $k$ (say $k\geq N$), we have 
	$$\boxed{\frac{3}{4}\|x^{k} - x^{k+1}\| \leq  \frac{1}{4}\|x^{k-1} - x^{k}\| + \frac{\max\{C_1,C_2\}}{D} \left(\varphi(\Psi(x^{k})) - \varphi(\Psi(x^{k+1}))\right).}$$
	Then, summing for $k$ from $N$ to $\infty$ and using the fact that $\varphi(\Psi(x^k))\to 0$ as $k\to \infty$, we get
	$$\frac{1}{2}\sum_{k=N}^{\infty} \|x^{k} - x^{k+1}\| \leq \frac{1}{4}\|x^{N-1} - x^{N}\| +   \frac{\max\{C_1,C_2\}}{D} \varphi(\Psi(x^{N})) < \infty,$$
	that is 
	$$\sum_{k=0}^{\infty} \|x^{k} - x^{k+1}\| < \infty,$$
	which implies that the sequence $\{x^k\}$ is Cauchy, thus convergent. \end{proof}
\paragraph*{Discussion}
\begin{itemize}
	\item The regularity assumption \Hthree\ is often an intrinsic property of $\Psi$, which can be derived from certain error bounds. We will explore this connection later in the proof of Lemma \ref{lem:H1-H3_P_Pc_KL}.
	\item In some variants of DCA, the term $D\|x^k - x^{k+1}\|^2$ in the sufficient descent assumption \Htwo\ may appear as $D\|x^{k-1} - x^k\|^2$. The convergence of $\{x^k\}$ still holds with a similar proof. 
\end{itemize}

\subsection{Convergence of DCA under the {\L}ojasiewicz Subgradient Inequality}
We now analyze the global convergence of the DCA iterates $\{x^k\}$ via Theorem~\ref{thm:generalcvx}, under the central tool of the \emph{{\L}ojasiewicz subgradient inequality}, originally established by Bolte–Daniilidis–Lewis~\cite[Theorem~3.1]{bolte2007lojasiewicz}. We begin by recalling this inequality, and then derive the convergence of $\{x^k\}$ in two settings: the standard DC program and the convex-constrained DC program.

\subsubsection{{\L}ojasiewicz Subgradient Inequality}

\begin{theorem}[{\L}ojasiewicz subgradient inequality, \cite{bolte2007lojasiewicz}]\label{thm:Loja-ineq}
	Let $f:\R^n\to (-\infty,\infty]$ be a subanalytic function with closed domain, and assume that $f$ is continuous on its domain. Let $x^\star \in \R^n$ be a limiting stationary point of $f$ (i.e., $0 \in \partial^L f(x^\star)$). Then there exist an exponent $\theta \in [0,1)$, a constant $M>0$, and a neighbourhood $\mathcal{V}$ of $x^\star$ such that 
	\begin{equation}\label{eq:Loja-ineq}
		|f(x)-f(x^\star)|^{\theta} \leq M \|y\|, \quad \forall x \in \mathcal{V},~ y \in \partial^L f(x),
	\end{equation}
	with the convention $0^0=1$.
\end{theorem}

The notions of subanalytic sets and functions are classical; see \cite{lojasiewicz1965ensembles,lojasiewicz1993geometrie,bierstone1988semianalytic,bolte2007lojasiewicz}. This class includes all analytic functions and enjoys rich stability properties (closure under finite unions, intersections, complements, and projections). Additionally, the distance function to a subanalytic set is subanalytic, and the sum or difference of continuous and subanalytic functions is subanalytic. Importantly, the {\L}ojasiewicz inequality provides a quantitative description of the geometry of $f$ near stationary points.

The symbol $\partial^L f(x)$ denotes the \emph{limiting subdifferential}, defined via the Fr\'echet subdifferential. We recall the relevant definitions below.

\begin{definition}[Fr\'echet subdifferential]\label{def:F-subdiff}
	For a proper closed function $f:\R^n\to (-\infty,\infty]$ and $x \in \dom f$, the Fr\'echet subdifferential is
	\[
	\partial^F f(x) := \left\{ y \in \R^n : \liminf_{\substack{z \to x \\ z \neq x}} \frac{f(z)-f(x)-\langle y,z-x\rangle}{\|z-x\|} \geq 0 \right\}.
	\]
	If $x \notin \dom f$, then $\partial^F f(x)=\emptyset$.
\end{definition}

\begin{definition}[Limiting subdifferential]\label{def:l-subdiff}
	For a proper closed function $f:\R^n\to (-\infty,\infty]$ and $x\in \R^n$, the limiting subdifferential is
	\[
	\partial^L f(x) := \{ y \in \R^n : \exists\, (x^k \to x, f(x^k)\to f(x), y^k \in \partial^F f(x^k)) \text{ with } y^k \to y\}.
	\]
\end{definition}

\begin{proposition}
	Let $f:\R^n\to (-\infty,\infty]$ be a proper closed function and $x\in\R^n$. Then:
	\begin{itemize}
		\item $\partial^F f(x) \subset \partial^L f(x)$; $\partial^F f(x)$ is closed and convex, while $\partial^L f(x)$ is closed but not necessarily convex.
		\item $\dom \partial^F f$ and $\dom \partial^L f$ are dense in $\dom f$ \cite{bolte2007lojasiewicz}.
		\item If $f$ is differentiable at $x$, then $\partial^F f(x)=\{\nabla f(x)\}$ and $\nabla f(x) \in \partial^L f(x)$. 
		\item If $f$ is $C^1$, then $\partial^F f(x)=\partial^L f(x)=\{\nabla f(x)\}$ for all $x$.
		\item If $f$ is convex and $x \in \operatorname{ri}(\dom f)$, then $\partial f(x)=\partial^F f(x)=\partial^L f(x)$, where $\partial f(x)$ is the convex subdifferential, i.e., $\partial f(x) := \{y \in \R^n : f(z) \geq f(x) + \langle y, z - x \rangle, \forall z \in \R^n\}$.
		\item $\nabla f(x)$ may not be unique in $\partial^L f(x)$ even for differentiable $f$. For instance, $f(x)=x^2\sin(1/x)$ for $x\neq 0$ and $f(0)=0$ is differentiable everywhere with $\partial^F f(0)=\{0\} \subset \partial^L f(0)=[-1,1]$.
	\end{itemize}
\end{proposition}

\begin{proposition}[Summation rules]
	Let $g,h:\R^n\to (-\infty,\infty]$ be proper closed functions and $x\in\R^n$. Then:
	\begin{itemize}
		\item If $h$ is $C^1$ at $x$ and $g$ is finite at $x$, then
		\[
		\partial^F(g+h)(x) = \partial^F g(x)+\nabla h(x), \qquad \partial^L(g+h)(x)=\partial^L g(x)+\nabla h(x).
		\]
		\item If $g,h \in \Gamma_0(\R^n)$ and $h$ is continuous at $x$, then
		\[
		\partial^F(g-h)(x) \subset \partial^L(g-h)(x) \subset \partial g(x)-\partial h(x).
		\]
		In particular:
		\begin{itemize}
			\item If $h$ is differentiable at $x$, then $$\partial^F(g-h)(x)=\partial^L(g-h)(x)=\partial g(x)-\nabla h(x).$$
			\item If $g$ is differentiable at $x$, then 
			\[
			\partial^L(g-h)(x)=\nabla g(x)+\partial^L(-h)(x) \subset \nabla g(x)-\partial h(x).
			\]
			The inclusion may be strict. For instance, with $h(x)=|x|$, we have $\partial^L(-h)(0)=\{-1,1\}\subset [-1,1]=-\partial h(0)$.
		\end{itemize}
	\end{itemize}
\end{proposition}
Background on variational analysis and subdifferential calculus can be found in Rockafellar and Wets~\cite{rockafellar2009variational} and Mordukhovich~\cite{mordukhovich2006variational}.

\subsubsection[Convergence for the Standard DC Program]{Convergence of $\{x^k\}$ for Standard DC Program}

\begin{theorem}\label{thm:globalcv-Lineq}
	Consider the standard DC program \eqref{prob:P} under \AssumpA. 
	Let $\{x^k\}$ be the DCA sequence starting from $x^0\in\dom\partial h$, and set $y^k\in\partial h(x^k)$ for each $k$. 
	Assume that $\{x^k\}$ (hence $\{y^k\}$) is well-defined and bounded. 
	Suppose further that:
	\begin{itemize}
		\item $f$ is continuous on $\dom f$;
		\item $\Phi:=f$ satisfies the \L{}ojasiewicz subgradient (KL) inequality at every cluster point of $\{x^k\}$;
		\item $\rho_g+\rho_h>0$, where $\rho_g$ and $\rho_h$ are the strong convexity moduli of $g$ and $h$, respectively;
		\item $h$ has a locally Lipschitz continuous gradient on a neighborhood of $\{x^k\}$ (so $y^k=\nabla h(x^k)$).
	\end{itemize}
	Then the whole sequence $\{x^k\}$ converges.
\end{theorem}

\begin{proof}
	If there exists $j\ge 0$ such that $x^{j+1}=x^j$, then $\{x^k\}$ converges in finitely many steps to $x^j$. 
	Otherwise, when $x^{k+1}\ne x^k$ for all $k$, we invoke Theorem~\ref{thm:generalcvx}. 
	By Lemma~\ref{lem:H1-H3_P}, the assumptions \Hone--\Hthree\ required by Theorem~\ref{thm:generalcvx} hold under the above conditions (boundedness of the iterates, $\rho_g+\rho_h>0$, local Lipschitz continuity of $\nabla h$, and the KL property of $\Phi$ at cluster points). 
	Hence $\{x^k\}$ is a Cauchy sequence and thus convergent.
\end{proof}

\begin{lemma}\label{lem:H1-H3_P}
	Under the assumptions of Theorem~\ref{thm:globalcv-Lineq}, if the sequence $\{x^k\}$ does not terminate in finitely many iterations (i.e., $x^{k+1}\neq x^k$ for all $k\in\mathbb N$), then the assumptions \Hone--\Hthree\ required in Theorem~\ref{thm:generalcvx} hold.
\end{lemma}

\begin{proof}
	\textbf{\Hone.}
	Let $\omega(x^0)$ be the set of cluster points of $\{x^k\}$. Since $\{x^k\}$ is bounded, $\omega(x^0)$ is nonempty, bounded, and closed, hence compact. 
	By Lemma~\ref{lemma:nonincreasingoffxk}, $\{f(x^k)\}$ is nonincreasing and convergent; therefore every cluster subsequence has the same limit, say $f^\star$. 
	By continuity of $f$ on $\dom f$, $f$ is constant on $\omega(x^0)$ with value $f^\star$. 
	Define the Lyapunov function $$\Psi(x):=\Phi(x)-f^\star=f(x)-f^\star.$$ Then
	\begin{equation}\label{eq:psi-nonneg-limit}
		\Psi(x^k)=f(x^k)-f^\star\ge 0,\qquad 
		\lim_{k\to\infty}\Psi(x^k)=\Psi(\omega(x^0))=0.
	\end{equation}
	Thus \Hone\ holds.
	
	\medskip
	\textbf{\Htwo.}
	By Lemma~\ref{lemma:suffdescent&squaresummable} (sufficient descent),
	\begin{equation}\label{eq:suff-descent}
		\Psi(x^k)-\Psi(x^{k+1}) \;=\; f(x^k)-f(x^{k+1})
		\;\ge\; \frac{\rho_g+\rho_h}{2}\,\|x^{k+1}-x^k\|^2, \forall k\ge 0.
	\end{equation}
	Hence \Htwo\ holds with $D=(\rho_g+\rho_h)/2>0$.
	
	\medskip
	\textbf{\Hthree.}
	First, by Theorem~\ref{thm:subseqconv}, any $x^\star\in\omega(x^0)$ is a DC critical point of $f=g-h$. 
	Since $h$ has a locally Lipschitz continuous gradient on a neighborhood of the trajectory, for $\Psi=f-f^\star$ we have the exact limiting-subdifferential identity
	\begin{equation}\label{eq:subdiff-identity}
		\partial^L \Psi(x) \;=\; \partial^L (g-h)(x) \;=\; \partial g(x)-\nabla h(x)\qquad \text{for all }x \text{ near }\omega(x^0).
	\end{equation}
	By assumption, $\Phi=f$ satisfies the \L{}ojasiewicz subgradient inequality at every cluster point. 
	By Theorem~\ref{thm:Loja-ineq}, for each $z\in\omega(x^0)$ there exist 
	$\theta_z\in[0,1)$, $M_z>0$, and a neighborhood $\mathcal V_z$ of $z$ such that
	$|\Psi(x)|^{\theta_z} \le M_z \|y\|$ for all $x\in\mathcal V_z$ and $y\in\partial^L\Psi(x)$.
	Since $\omega(x^0)$ is compact, pick a finite subcover $\{\mathcal V_{z_i}\}_{i=1}^p$.
	By continuity and $\Psi(\omega(x^0))=0$, shrink the neighborhoods if needed so that
	$|\Psi(x)|<1$ on $\bigcup_{i=1}^p \mathcal V_{z_i}$. 
	Moreover, choose $\epsilon>0$ such that $B(z_i,\epsilon)\subset \mathcal V_{z_i}$ for all $i$.
	Set $\theta:=\max_i\theta_{z_i}$ and $M:=\max_i M_{z_i}$. 
	Then for every $x\in\bigcup_{i=1}^p B(z_i,\epsilon)$ and $y\in\partial^L\Psi(x)$,
	since $0\le|\Psi(x)|<1$ we have $|\Psi(x)|^{\theta}\le|\Psi(x)|^{\theta_{z_i}}$, hence
	\[
	|\Psi(x)|^{\theta} \le M \|y\|.
	\]
	As $\mathrm{dist}(x^k,\omega(x^0))\to 0$ and $\bigcup_{i=1}^p B(z_i,\epsilon)$ is an open neighborhood of $\omega(x^0)$, there exists $N_0$ such that $x^k\in\bigcup_{i=1}^p B(z_i,\epsilon)$ for all $k\ge N_0$. Therefore,
	\begin{equation}\label{eq:L-ineqxk}
		|\Psi(x^k)|^{\theta} \le M \|y\|,\qquad \forall k\ge N_0,\ \forall y\in\partial^L\Psi(x^k).
	\end{equation}
	Next we link $\partial^L\Psi(x^{k+1})$ to the DCA step. 
	The first-order optimality condition for the convex subproblem defining $x^{k+1}$ gives
	\[
	\nabla h(x^k)\in \partial g(x^{k+1}).
	\]
	Therefore, by \eqref{eq:subdiff-identity},
	\begin{equation}\label{eq:elementinpartialpsiwithh}
		\nabla h(x^k)-\nabla h(x^{k+1}) \;\in\; \partial g(x^{k+1})-\nabla h(x^{k+1})
		\;=\; \partial^L\Psi(x^{k+1}).
	\end{equation}
	Since $\{x^k\}\cup \omega(x^0)$ is contained in some bounded open set $\mathcal D$ and $\nabla h$ is locally Lipschitz, there exists a uniform $L>0$ such that
	\begin{equation}\label{eq:unif-L}
		\|\nabla h(x)-\nabla h(z)\|\;\le\; L\,\|x-z\|\qquad \forall\, x,z\in\overline{\mathcal D}.
	\end{equation}
	Combining \eqref{eq:L-ineqxk}–\eqref{eq:elementinpartialpsiwithh}–\eqref{eq:unif-L}, there exists $N\ge N_0+1$ such that for all $k\ge N$,
	\begin{equation}\label{eq:KL-bound-diff}
		|\Psi(x^{k})|^\theta 
		\;\le\; M\, \bigl\|\nabla h(x^{k-1})-\nabla h(x^{k})\bigr\|
		\;\le\; ML\,\|x^{k}-x^{k-1}\|.
	\end{equation}
	Finally, let $\varphi(t)=t^{1-\theta}$ for $t\ge 0$. Then $\varphi$ is $C^1$, concave, $\varphi(0)=0$, and $\varphi'(t)=(1-\theta)t^{-\theta}>0$ for $t>0$. 
	By concavity and monotonicity of $\Psi(x^k)$,
	\begin{equation}\label{eq:concavity-step}
		\varphi(\Psi(x^k))-\varphi(\Psi(x^{k+1}))
		\;\ge\; \varphi'(\Psi(x^k))\bigl(\Psi(x^k)-\Psi(x^{k+1})\bigr).
	\end{equation}
	Using \eqref{eq:suff-descent} and \eqref{eq:KL-bound-diff}, for all $k\ge N$ we get
	\[
	\varphi'(\Psi(x^k))
	\;=\; (1-\theta)\Psi(x^k)^{-\theta}
	\;\ge\; \frac{1-\theta}{ML}\,\frac{1}{\|x^{k}-x^{k-1}\|}.
	\]
	Therefore,
	\[
	\varphi'(\Psi(x^k)) \times \frac{ML}{1-\theta}\,\|x^{k}-x^{k-1}\|\;\ge\; 1,\qquad \forall k\ge N,
	\]
	which establishes \Hthree\ with $C_1=0$ and $C_2=ML/(1-\theta)$.
\end{proof}

\paragraph*{Discussion}
\begin{itemize}
	\item If we carefully choose $y^k \in \partial h(x^k)$ in DCA so that
	\begin{equation}
		\label{eq:betteryk}
		\boxed{-y^k \in \partial^L(-h)(x^k)\subset -\partial h(x^k)}
	\end{equation}
	whenever $h$ is non-differentiable at $x^k$, then we can achieve a similar result in \Hthree\ by assuming that ``g has a locally Lipschitz continuous gradient.'' Note that at non-differentiable points, $\partial^{L}(-h)(x^k)$ is typically a proper subset of $-\,\partial h(x^k)$; choosing $-y^k\in\partial^{L}(-h)(x^k)$ avoids interior selections (e.g., $0$ in the $\ell_1$ case) that may fail to be limiting. This case is particularly interesting in practice, as it allows for the non-differentiability of $h$ at some points $x^k$. The proof of \Hthree\ in this scenario proceeds as follows:
	
	If $g$ has a locally Lipschitz continuous gradient over the bounded open set $\setD$ (containing the sequence $\{x^k\}$), then
	\begin{equation}\label{eq:L-smoothofg}
		\exists L > 0, \forall k \geq 1, \|\nabla g(x^k) - \nabla g(x^{k+1})\| \leq L \|x^k - x^{k+1}\|.
	\end{equation}
	The first-order optimality condition for the convex subproblem \eqref{prob:Pk} gives us 
	$$ -\nabla g(x^{k+1}) = -y^k \in \partial^L (-h)(x^k),$$
	which implies 
	$$\nabla g(x^k) - \nabla g(x^{k+1}) = \nabla g(x^k) - y^k \in \nabla g(x^k) + \partial^L (-h)(x^k) = \partial^L \Psi(x^k).$$
	Therefore,
	\begin{equation}
		\label{eq:elementinpartialpsi}
		\forall k \geq 1, \nabla g(x^k) - \nabla g(x^{k+1}) \in \partial^L \Psi(x^k).
	\end{equation}
	From \eqref{eq:L-ineqxk} and \eqref{eq:elementinpartialpsi}, we obtain
	$$\exists (N > 0, M > 0, \theta \in [0,1)), \forall k \geq N, |\Psi(x^k)|^{\theta} \leq M \|\nabla g(x^k) - \nabla g(x^{k+1})\|.$$
	Combining this with \eqref{eq:L-smoothofg}, we have
	\begin{equation}
		\label{eq:L-ineqatxkwithg}
		\boxed{\exists (N > 0, M > 0, \theta \in [0,1), L > 0), \forall k \geq N, |\Psi(x^k)|^{\theta} \leq ML \|x^k - x^{k+1}\|.}
	\end{equation}
	Hence, for all $k \geq N$, we get
	$$\varphi'(\Psi(x^k)) = (1 - \theta)\Psi(x^k)^{-\theta} \overset{\eqref{eq:L-ineqatxkwithg}}{\geq} \frac{1 - \theta}{ML\|x^k - x^{k+1}\|},$$
	which implies 
	$$\boxed{\varphi'(\Psi(x^k)) \times \frac{ML}{1 - \theta} \|x^k - x^{k+1}\| \geq 1,}$$
	satisfying \Hthree\ with $C_1 = ML/(1 - \theta)$ and $C_2 = 0$.
	
	\item If, for all $x^k$ ($k \geq 1$), either $g$ or $h$ has a locally Lipschitz continuous gradient around $x^k$, then by combining \eqref{eq:elementinpartialpsiwithh} and \eqref{eq:elementinpartialpsi}, we have $\exists (N > 0, M > 0, \theta \in [0,1), L > 0)$ such that for all $k \geq N + 1$,
	$$|\Psi(x^k)|^{\theta} \leq ML (\|x^k - x^{k+1}\| + \|x^{k-1} - x^k\|).$$
	Hence, for all $k \geq N + 1$,
	$$\varphi'(\Psi(x^k)) = (1 - \theta)\Psi(x^k)^{-\theta} \geq \frac{1 - \theta}{ML(\|x^k - x^{k+1}\| + \|x^{k-1} - x^k\|)},$$
	which implies 
	$$\boxed{\varphi'(\Psi(x^k)) \times \frac{ML}{1 - \theta} (\|x^k - x^{k+1}\| + \|x^{k-1} - x^k\|) \geq 1,}$$
	verifying \Hthree\ with $C_1 = C_2 = ML/(1 - \theta)$.
\end{itemize}

\subsubsection[Convergence for the Convex-Constrained DC Program]{Convergence of $\{x^k\}$ for Convex Constrained DC Program}
\begin{theorem}\label{thm:globalcv-Pc}
	Consider the convex constrained DC program \eqref{prob:P} under \AssumpA. Let $\{x^k\}$ and $\{y^k\}=\{\nabla h(x^k)\}$ be two well-defined and bounded sequences generated by DCA starting from an initial point $x^0 \in \dom \partial h$ for problem \eqref{prob:P}. Moreover, suppose that:
	\begin{itemize}
		\item $f$ is continuous over $\dom f$;
		\item $\Phi(x) := f(x) + \chi_{\setC}(x)$ satisfies the {\L}ojasiewicz subgradient inequality at any cluster point of $\{x^k\}$;
		\item $\rho_g + \rho_h > 0$;
		\item $h$ has a locally Lipschitz continuous gradient over $\setC$.
	\end{itemize}
	Then the sequence $\{x^k\}$ is convergent.
\end{theorem}

Due to the equivalent standard DC formulation of \eqref{prob:P} by introducing the indicator function $\chi_{\setC}$, the proof of Theorem \ref{thm:globalcv-Pc} follows the same reasoning as in Theorem \ref{thm:globalcv-Lineq}. We only need to verify the assumptions \Hone-\Hthree\ required in Theorem \ref{thm:generalcvx}, which will be discussed in Lemma \ref{lem:H1-H3_Pc}.

\begin{lemma}\label{lem:H1-H3_Pc}
	Under the assumptions of Theorem \ref{thm:globalcv-Pc}, if the sequence $\{x^k\}$ does not converge in finitely many iterations, then the assumptions \Hone-\Hthree\ required in Theorem \ref{thm:generalcvx} hold.
\end{lemma}

\begin{proof}
	Let $\Psi(x):=\Phi(x) - f^\star$. \Hone\ and \Htwo\ can be verified exactly as in Lemma \ref{lem:H1-H3_P}. Here, we will only highlight the differences in the proof of \Hthree. We first derive the formulas \eqref{eq:L-ineqxk} and \eqref{eq:concavity-step}, and show that there exists a bounded open set $\setD \subset \setC$ containing the sequence $\{x^k\}$ and the bounded set $\omega(x^0)$, over which $h$ has a locally Lipschitz continuous gradient, i.e.,   
	\begin{equation}\label{eq:L-smoothofh-bis}
		\exists L > 0, \forall k \geq 1, \|\nabla h(x^k) - \nabla h(x^{k+1})\| \leq L \|x^k - x^{k+1}\|.
	\end{equation}
	The first-order optimality condition for the convex subproblem \eqref{prob:Pk} (under constraint $\setC$) gives 
	$$ \nabla h(x^k) \in \partial (g + \chi_{\setC})(x^{k+1}).$$
	Thus,
	$$\nabla h(x^k) - \nabla h(x^{k+1}) \in \partial (g + \chi_{\setC})(x^{k+1}) - \nabla h(x^{k+1}) = \partial^L \Psi(x^{k+1}).$$
	Hence,
	\begin{equation}
		\label{eq:elementinpartialpsiwithh-bis}
		\forall k \geq 1, \nabla h(x^k) - \nabla h(x^{k+1}) \in \partial^L \Psi(x^{k+1}).
	\end{equation}
	From this, we again obtain for each $k \geq N + 1$ that
	$$\varphi'(\Psi(x^k)) \times \frac{ML}{1-\theta} \|x^{k-1} - x^k\| \geq 1,$$
	which satisfies \Hthree\ with $C_1 = 0$ and $C_2 = ML/(1-\theta)$.
\end{proof}

\begin{remark}
	Additional assumptions are required to guarantee the global convergence of $\{x^k\}$ when $g$ has a locally Lipschitz continuous gradient in the convex constrained DC program. Specifically, it is necessary to choose $y^k \in \partial h(x^k)$ and $x^{k+1} \in \argmin\{g(x) - \langle y^k, x \rangle : x \in \setC\}$ such that $-\nabla g(x^{k+1}) \in \partial^L (\chi_{\setC} - h)(x^k)$. However, verifying this condition in practice can be challenging.
\end{remark}

\subsection{Convergence of DCA under the KL Property}\label{sec:cvofdcawithKL}
In this section, we focus on the global convergence of the sequence $\{x^k\}$ generated by DCA, using Theorem \ref{thm:generalcvx} under the Kurdyka--{\L}ojasiewicz property. We begin by recalling the Kurdyka--{\L}ojasiewicz property, and then establish the convergence of $\{x^k\}$.

\subsubsection{Kurdyka--{\L}ojasiewicz Property}\label{subsec:KL}
\begin{definition}[Kurdyka--{\L}ojasiewicz Property]\label{thm:KL-ineq}
	Let $f:\R^n\to \R$ be a locally Lipschitz function. We say that $f$ satisfies the \emph{KL property} at $x^\star \in \R^n$ if there exist $\eta \in (0,\infty]$, a neighborhood $\mathcal{V}$ around $x^\star$, and a concave function $\varphi:[0,\eta)\to \R_+$ such that:
	\begin{itemize}
		\item $\varphi(0) = 0$ and $\varphi \in \mathcal{C}^1((0,\eta),\R_+)$;
		\item $\varphi' > 0$ on $(0,\eta)$;
		\item For all $x \in \mathcal{V}$ with $f(x) - f(x^\star) \in (0,\eta)$, we have the KL property:
		\begin{equation}
			\label{eq:KL-ineq-def}
			\varphi'(f(x) - f(x^\star)) \times \text{dist}(0,\partial^L f(x)) \geq 1.
		\end{equation}
	\end{itemize}
	We say that $f$ is \emph{a KL function} if it satisfies the KL inequality over $\dom \partial^L f$.
\end{definition} 

The KL property was originally developed by {\L}ojasiewicz \cite{lojasiewicz1963propriete} (cf. {\L}ojasiewicz inequality) for differentiable subanalytic functions and was later generalized by Kurdyka \cite{kurdyka1998gradients} to definable (cf. tame \cite{van1998tame}) functions. This concept was further extended to the nonsmooth setting by Bolte et al. \cite{bolte2007lojasiewicz,attouch2013convergence}, where the gradient is replaced by the limiting subdifferential. 

A remarkable aspect of KL functions is their ubiquity in applications; for example, semialgebraic, subanalytic, and log-exp functions are KL functions (see \cite{kurdyka1998gradients,bolte2007lojasiewicz,attouch2013convergence} and the references therein).

In practice, we often choose $\varphi(t) = M t^{1-\theta}$ for some $M > 0$ and $\theta \in [0,1)$, where $\theta$ is called the {\L}ojasiewicz exponent. Then, for all $x \in \mathcal{V}$ with $f(x) - f(x^\star) \in (0,\eta)$, we have:
$$(f(x) - f(x^\star))^{\theta} \leq M(1-\theta) \text{dist}(0,\partial^L f(x)),$$
which corresponds exactly to the {\L}ojasiewicz subgradient inequality (given in Theorem \ref{thm:Loja-ineq}) by taking $\text{dist}(0,\partial^L f(x)) := \min\{\|y\| : y \in \partial^L f(x)\}$ and letting $0 \in \partial^L f(x^\star)$. Hence, the {\L}ojasiewicz subgradient inequality can be viewed as a special case of the KL property.

\subsubsection[Convergence of the Iterates under the KL Property]{Convergence of $\{x^k\}$ under the KL Property}
\begin{theorem}\label{thm:globalcv-KL}
	Consider the standard DC program \eqref{prob:P} (or the convex constrained DC program \eqref{prob:P}) under \AssumpA. Let $\{x^k\}$ and $\{y^k\} \subset \{\partial h(x^k)\}$ be two well-defined and bounded sequences generated by DCA, starting from an initial point $x^0 \in \dom \partial h$ for problem \eqref{prob:P} (or \eqref{prob:P}).  
	Moreover, suppose that:
	\begin{itemize}
		\item $f$ is continuous over $\dom f$;
		\item $\Phi(x) := f(x)$ (or $\Phi(x) := f(x) + \chi_{\setC}(x)$) is a KL function;
		\item $\rho_g + \rho_h > 0$;
		\item $h$ has a locally Lipschitz continuous gradient.
	\end{itemize}
	Then, the sequence $\{x^k\}$ is convergent.
\end{theorem}

As before, Theorem \ref{thm:globalcv-KL} can be proved by applying Theorem \ref{thm:generalcvx}, with the assumptions \Hone-\Hthree\ being verified in Lemma \ref{lem:H1-H3_P_Pc_KL}.

\begin{lemma}\label{lem:H1-H3_P_Pc_KL}
	Under the assumptions of Theorem \ref{thm:globalcv-KL}, if the sequence $\{x^k\}$ does not converge in finitely many iterations, then the assumptions \Hone-\Hthree\ required in Theorem \ref{thm:generalcvx} hold.
\end{lemma}

\begin{proof}
	\Hone\ and \Htwo\ do not depend on the KL property, so they are proved exactly as in Lemma \ref{lem:H1-H3_P}. Now, we only need to prove \Hthree\ as follows: Let $\setC = \R^n$ for the standard DC program. Then, for both the standard and convex constrained DC programs, $\Psi := \Phi - f^\star$ is a KL function since $\Phi$ is so. This implies that there exist $\eta \in (0,\infty]$, $\epsilon > 0$, a neighborhood $\mathcal{V}$ around $\omega(x^0)$ (defined as $\mathcal{V} := \cup_{z \in \omega(x^0)} B(z,\epsilon)$), and a concave function $\varphi: [0,\eta) \to \R_+$ with  
	$\varphi(0) = 0$, $\varphi \in \mathcal{C}^1((0,\eta),\R_+)$, and $\varphi' > 0$ on $(0,\eta)$, such that for all $x \in \mathcal{V}$ and $\Psi(x) \in (0,\eta)$, we have:
	\begin{equation}
		\label{eq:KL-ineq}
		\varphi'(\Psi(x)) \times \text{dist}(0, \partial^L \Psi(x)) \geq 1.
	\end{equation}
	Since $\omega(x^0)$ is the set of limit points of the sequence $\{x^k\}$, it follows that
	\begin{equation}
		\label{eq:boundednessofxk-KL}
		\exists N > 0, \forall k \geq N, x^k \in \mathcal{V} \text{ and } \Psi(x^k) \in (0,\eta).
	\end{equation}
	Combining \eqref{eq:KL-ineq} and \eqref{eq:boundednessofxk-KL}, we obtain:  
	\begin{equation}
		\label{eq:L-ineqxk-KL}
		\exists N > 0, \forall k \geq N, \varphi'(\Psi(x^k)) \times \text{dist}(0, \partial^L \Psi(x^k)) \geq 1.
	\end{equation}
	Due to the boundedness of the sequence $\{x^k\}$ and the set $\omega(x^0)$, there exists a bounded open set $\setD \subset \setC$ containing the entire sequence $\{x^k\}$ and $\omega(x^0)$. Then, by assumption, $\nabla h$ is locally Lipschitz on $\mathcal D \subset \setC$, hence 
	\begin{equation}\label{eq:L-smoothofh-KL}
		\exists L > 0, \forall k \geq 1, \|\nabla h(x^k) - \nabla h(x^{k+1})\| \leq L \|x^k - x^{k+1}\|.
	\end{equation}
	The first-order optimality condition for the convex subproblem \eqref{prob:Pk} (under the constraint $\setC$) gives: 
	$$ \nabla h(x^k) \in \partial(g + \chi_{\setC})(x^{k+1}).$$
	Thus,
	$$\nabla h(x^k) - \nabla h(x^{k+1}) \in \partial(g + \chi_{\setC})(x^{k+1}) - \nabla h(x^{k+1}) = \partial^L \Psi(x^{k+1}),$$
	which implies the error bound:
	\begin{equation}
		\label{eq:elementinpartialpsiwithh-KL}
		\forall k \geq N+1, \text{dist}(0, \partial^L \Psi(x^k)) \leq \|\nabla h(x^{k-1}) - \nabla h(x^k)\|.
	\end{equation}
	From \eqref{eq:L-ineqxk-KL}, \eqref{eq:L-smoothofh-KL}, and \eqref{eq:elementinpartialpsiwithh-KL}, it follows that:
	\begin{equation}
		\label{eq:L-ineqatxkwithh-KL}
		\boxed{\exists N > 0, \forall k \geq N+1, \varphi'(\Psi(x^k)) \times L \|x^{k-1} - x^k\| \geq 1,}
	\end{equation}
	which satisfies \Hthree\ with $C_1 = 0$ and $C_2 = L$.
\end{proof}

\paragraph*{Discussion} To the case where $g$ is differentiable and $\nabla g$ is locally Lipschitz, we can verify \Hthree\ via the following lemma:

\begin{lemma}\label{lem:iRE}
	Let $\Phi:=f+\chi_{\mathcal C}$ with a DC function $f=g-h$, where $g$ is differentiable, $\nabla g$ is locally Lipschitz on a neighborhood of the trajectory $\{x^k\}$ and $y^k\in\partial h(x^k)$. Assume there exists $B>0$ such that
	$$\|y^{k+1}-y^k\|\le B\,\|x^{k+1}-x^k\|,\quad \text{for all sufficiently large  } k.$$
	Then
	$$\mathrm{dist}\!\big(0,\partial^L\Phi(x^{k+1})\big)\le B\,\|x^{k+1}-x^k\|.$$
	In particular, one may take $-\,y^{k+1}\in\partial^{L}(-h)(x^{k+1})$.
\end{lemma}

\begin{proof}
	Let $\Delta x^k:=x^{k+1}-x^k$. 
	DCA computes $x^{k+1}$ as a minimizer of
	\[
	\min_{x\in\mathcal C}\; g(x)-\langle y^k,x\rangle
	\quad\Longleftrightarrow\quad
	0\in \nabla g(x^{k+1})-y^k+N_{\mathcal C}(x^{k+1}).
	\]
	Hence there exists $u^{k+1}\in N_{\mathcal C}(x^{k+1})$ such that
	\begin{equation}\label{eq:KKT}
		y^k=\nabla g(x^{k+1})+u^{k+1}.
	\end{equation}
	Define
	\[
	v^{k+1}:=\nabla g(x^{k+1})+u^{k+1}-y^{k+1}=y^k-y^{k+1}.
	\]
	Since $y^{k+1}\in\partial h(x^{k+1})$ and $u^{k+1}\in N_{\mathcal C}(x^{k+1})=\partial \chi_{\mathcal C}(x^{k+1})$, and $g$ is smooth, we have
	\[
	v^{k+1}\ \in\ \nabla g(x^{k+1})+N_{\mathcal C}(x^{k+1})-\partial h(x^{k+1}).
	\]
	For the limiting subdifferential one has the standard sum rule
	$\partial^{L}(g+\chi_{\mathcal C})(x)=\nabla g(x)+N_{\mathcal C}(x)$ and
	\[
	\partial^{L}\Phi(x)=\partial^{L}\big((g+\chi_{\mathcal C})-h\big)(x)\ \subset\ \partial(g+\chi_{\mathcal C})(x)-\partial h(x).
	\]
	Thus $v^{k+1}$ at least belongs to $\partial(g+\chi_{\mathcal C})(x^{k+1})-\partial h(x^{k+1})$. 
	If, in addition, one chooses a \emph{limiting} selection for $-h$ (e.g., $-\,y^{k+1}\in\partial^{L}(-h)(x^{k+1})$), then
	\[
	v^{k+1}\in \partial(g+\chi_{\mathcal C})(x^{k+1})+\partial^{L}(-h)(x^{k+1}) = \partial^L \Phi(x^{k+1}).
	\]
	Hence, for all large enough $k$
	\[\mathrm{dist}\!\big(0,\partial^L\Phi(x^{k+1})\big)\le \|v^{k+1}\| = \|y^{k+1}-y^k\|\le B\,\|x^{k+1}-x^k\|.\]
\end{proof}

\subsection{Relation to Classical KL-Based Frameworks}\label{subsec:positioning}

Classical KL-based frameworks for descent methods, typified by Attouch--Bolte--Svaiter \cite{attouch2013convergence}, posit a Lyapunov energy $\Phi$ obeying the KL property at cluster points of $\{x^k\}$ and derive convergence from a \emph{sufficient decrease} (SD) plus a \emph{relative error} (RE) bound, under mild boundedness/continuity assumptions. Before that machinery can be applied to DCA, however, three DC-specific issues must be addressed: well-definedness depends on the chosen decomposition, the residual estimate depends on the subgradient selection, and the natural notion of DC criticality is weaker than standard stationarity. We keep the KL backbone but change the \emph{assumption shape} and \emph{verification route} so that they are native to DCA and easy to port to its variants:

\begin{enumerate}
	\item \textbf{Energy flexibility \Hone.} We work with a general Lyapunov $\Psi$ (not necessarily $f$). This covers proximal/linearized, inertial, and boosted DCA variants where descent holds for a modified energy rather than $f$.
	
	\item \textbf{Descent from the subproblem \Htwo.} The quadratic decrease $\Psi(x^k)-\Psi(x^{k+1})\ge D\|x^{k+1}-x^k\|^2$ is a direct consequence of convex subproblem optimality, or of adding a standard regularizer. Hence $D$ is \emph{read off} from the model, with no appeal to global Lipschitz constants on $f$.
	
	\item \textbf{RE for \Hthree.}
	Instead of committing to the classical distance bound 
	$\mathrm{dist}(0,\partial\Phi(x^{k+1}))\le B\|x^{k+1}-x^k\|$, 
	we allow either of two interchangeable forms that are native to DCA:
	(i) a distance bound obtained by combining the KL inequality with the subproblem optimality (this is the route used in Lemma \ref{lem:H1-H3_P_Pc_KL}); or 
	(ii) a distance bound derived from the monotonicity of $\partial h$ and the local Lipschitz continuity of $\nabla g$ along the trajectory (see Lemma~\ref{lem:iRE}). 
	Both yield \Hthree\ after composing with the KL desingularizing function $\varphi$.

	\item \textbf{Well-definedness is separated.} We explicitly front-load \AssumpA\ and \AssumpB\ (existence/solvability of subproblems) before convergence. This matters in DC: natural decompositions may break DCA, and the fix (tuning the decomposition or adding $\varphi$) is orthogonal to KL.
	
	\item \textbf{Error-tolerance and variants.} The \Htwo-\Hthree\ format is stable under inexact solves and noisy $y^k$ (adding summable errors), and it is \emph{plug-and-play} across DCA variants.
\end{enumerate}

To make the last point concrete, consider a proximal or linearized DCA variant in which the convex subproblem is regularized by a quadratic term. The added term modifies the natural Lyapunov function but immediately strengthens \Htwo, while the optimality condition of the regularized subproblem yields the same type of residual estimate used in Lemma \ref{lem:H1-H3_P_Pc_KL}. This is a typical situation where the abstract KL philosophy remains valid, yet our DCA-specific verification route is what makes the proof reusable.

\subsection{Convergence Rates of DCA}\label{sec:convergencerate}
In this part, we focus on the convergence rates of DCA concerning the sequences $\{f(x^k)\}$ and $\{\|x^k - x^\star\|\}$ for DC programs under the KL property. Similar results will also hold under the {\L}ojasiewicz subgradient inequality, which is a special case of the KL property.

\begin{theorem}[Convergence Rate of $\{f(x^k)\}$]\label{thm:cvrate-DCA-KL-fk}
	Under the assumptions of Theorem \ref{thm:globalcv-KL}, suppose that $f(x^k) \to f^\star$ and that $\Psi$ satisfies the KL property with the concave function $\varphi(t) = M t^{1-\theta}$ for some $M > 0$ and $\theta \in [0,1)$. Then we have:
	\begin{enumerate}
		\item[(i)] If $\theta = 0$, then $\{f(x^k)\}$ converges to $f^\star$ in finitely many iterations.
		\item[(ii)] If $\theta \in (0,\frac{1}{2}]$, then there exists a constant $q \in (0,1)$ such that 
		\begin{equation}
			\label{eq:cvord-linear}
			f(x^k) - f^\star \leq O(q^k), \quad \text{as } k \to \infty;
		\end{equation}
		i.e., the sequence $\{f(x^k)\}$ converges linearly to $f^\star$.
		\item[(iii)] If $\theta \in (\frac{1}{2},1)$, then 
		\begin{equation}
			\label{eq:cvord-sublinear}
			f(x^k) - f^\star \leq O\left(k^{\frac{1}{1-2\theta}}\right), \quad \text{as } k \to \infty;
		\end{equation}
		i.e., the sequence $\{f(x^k)\}$ converges sublinearly to $f^\star$.
	\end{enumerate}
\end{theorem}
\begin{proof}
	By the definition of $\varphi$, we have:
	\begin{equation}
		\label{eq:diffphi-fk}
		\varphi'(\Psi(x^k)) = M(1-\theta)\Psi(x^k)^{-\theta}.
	\end{equation}
	Substituting $\varphi'(\Psi(x^k))$ into \eqref{eq:L-ineqatxkwithh-KL} and squaring both sides, we obtain for sufficiently large $k$:
	$$\Psi(x^k)^{2\theta} \leq M^2L^2(1-\theta)^2\|x^{k-1} - x^k\|^2 \overset{\Htwo}{\leq} \frac{M^2L^2(1-\theta)^2}{D} (\Psi(x^{k-1}) - \Psi(x^k)),$$
	where $D = (\rho_g + \rho_h)/2 > 0$ as given in \eqref{eq:suff-descent}. Now, setting $\alpha = 2\theta > 0$, $\beta = \frac{2M^2L^2(1-\theta)^2}{\rho_g + \rho_h} > 0$, and $r_k = \Psi(x^k)$, we obtain the desired convergence rate by applying Lemma \ref{lem:cvrate-bis}.
\end{proof}

\begin{remark}
	Similar convergence rates can be obtained for the case where $g$ has a locally Lipschitz continuous gradient, using Lemmas \ref{lem:cvrate} and \ref{lem:cvrate-tri}.
\end{remark}

\begin{theorem}[Convergence Rate of $\{\|x^k-x^\star\|\}$]\label{thm:cvrate-DCA-KL-xk}
	Under the assumptions of Theorem \ref{thm:globalcv-KL}, let $x^k \to x^\star$. Then:
	\begin{enumerate}
		\item[(i)] If $\theta = 0$, then the sequence $\{\|x^k - x^\star\|\}$ converges to $0$ in finitely many iterations.
		\item[(ii)] If $\theta \in (0,\frac{1}{2}]$ and $\|x^k - x^\star\| > 0$ for all $k \in \N$, then there exists a constant $q \in (0,1)$ such that 
		$$\|x^k - x^\star\| \leq O(q^k), \quad \text{as } k \to \infty;$$
		i.e., the sequence $\{\|x^k - x^\star\|\}$ converges linearly to $0$.
		\item[(iii)] If $\theta \in (\frac{1}{2},1)$ and $\|x^k - x^\star\| > 0$ for all $k \in \N$, then
		$$\|x^k - x^\star\| \leq O\left(k^{\frac{1-\theta}{1-2\theta}}\right), \quad \text{as } k \to \infty;$$
		i.e., the sequence $\{\|x^k - x^\star\|\}$ converges sublinearly to $0$.
	\end{enumerate}
\end{theorem}

\begin{proof}
	$(i)$ If $\theta=0$, Theorem \ref{thm:cvrate-DCA-KL-fk} implies that the sequence $\{f(x^k)\}$ converges to $f^\star$ in a finite number of iterations (say $T$ iterations). From the sufficient descent property \eqref{eq:suffdecprop}, i.e., 
	$$f(x^k) - f(x^{k+1}) \geq \frac{\rho_g+\rho_h}{2}\|x^k - x^{k+1}\|^2, \quad \forall k\in \N,$$
	we have 
	$$\|x^k - x^{k+1}\| = 0, \quad \forall k \geq T.$$
	Thus, $x^k = x^T$ for all $k \geq T$, implying that $\{x^k\}$ converges to $x^\star$ ($= x^T$) in $T$ iterations. For the case where $\theta \neq 0$, we must have $\|x^k - x^\star\| > 0$ for all $k \in \N$.\\
	$(ii)-(iii)$: Consider the residual $R_t := \sum_{k \geq t}\|x^k - x^{k+1}\|$ for all $t \in \N$. The sequence $\{R_t\}$ is non-negative, non-increasing, and converging to $0$. From 
	\begin{equation}\label{eq:resinSt}
		\|x^t - x^{t+1}\| = R_t - R_{t+1},
	\end{equation}
	and by the triangle inequality, we have
	\begin{equation}
		\label{eq:ineqresvsSt}
		\begin{aligned}
			\|x^t - x^\star\|
			&= \|x^t - x^{t+1} + x^{t+1} - x^{t+2} + \cdots + x^{k-1} - x^\star\| \\
			&\leq \sum_{k \geq t}\|x^k - x^{k+1}\| = R_t, \quad \forall t \in \N.
		\end{aligned}
	\end{equation} 
	Now, recall the relation \eqref{eq:residualrelation}, which states that there exists $N > 0$ such that for all $k \geq N+1$:
	$$\frac{3}{4}\|x^k - x^{k+1}\| \leq \frac{1}{4}\|x^{k-1} - x^k\| + \frac{\max\{C_1,C_2\}}{D} \left(\varphi(\Psi(x^k)) - \varphi(\Psi(x^{k+1}))\right).$$
	Summing for $k$ from $t \geq N+1$ to $\infty$ and using the fact that $\varphi(\Psi(x^k)) \to 0$ as $k \to \infty$, we get
	\begin{equation}\label{eq:rec_St}
		\frac{1}{2} R_t \leq \frac{1}{4}\|x^{t-1} - x^t\| + \frac{\max\{C_1,C_2\}}{D} \varphi(\Psi(x^t)), \quad \forall t \geq N+1.
	\end{equation}
	By the definition of $\varphi(s) = M s^{1-\theta}$ for some $M > 0$ and $\theta \in [0,1)$, we have
	\begin{equation}
		\label{eq:diffphi-xk}
		\varphi'(\Psi(x^k)) = M(1-\theta)\Psi(x^k)^{-\theta}.
	\end{equation}
	Recalling the relation \eqref{eq:L-ineqatxkwithh-KL}, for sufficiently large $k$, we have 
	\begin{equation}\label{eq:KL-ineqatxkwithh}
		\varphi'(\Psi(x^k)) \times L \|x^{k-1} - x^k\| \geq 1.
	\end{equation}
	From \eqref{eq:diffphi-xk} and \eqref{eq:KL-ineqatxkwithh}, we obtain
	\begin{equation}
		\label{eq:ineq_psi(xk)^theta}
		\Psi(x^k)^{\theta} \leq M L (1-\theta) \|x^k - x^{k-1}\|.
	\end{equation}
	Therefore,
	\begin{eqnarray*}
		\varphi(\Psi(x^k)) &=& M \Psi(x^k)^{1-\theta} \\
		&\overset{\eqref{eq:ineq_psi(xk)^theta}}{\leq}& M \left(M (1-\theta) L \|x^k - x^{k-1}\|\right)^{\frac{1-\theta}{\theta}} \\
		&\overset{\eqref{eq:resinSt}}{=}& M \left(M (1-\theta) L\right)^{\frac{1-\theta}{\theta}} (R_{t-1} - R_t)^{\frac{1-\theta}{\theta}}.
	\end{eqnarray*}
	Injecting this inequality and \eqref{eq:resinSt} into \eqref{eq:rec_St}, we obtain for sufficiently large $t$:
	\begin{equation}\label{eq:rec_St_new}
		\frac{1}{2} R_t \leq \frac{1}{4}(R_{t-1} - R_t) + \frac{\max\{C_1,C_2\}}{D}  M \left(M (1-\theta) L\right)^{\frac{1-\theta}{\theta}} (R_{k-1} - R_k)^{\frac{1-\theta}{\theta}}.
	\end{equation}
	Let $a := \frac{\max\{C_1,C_2\}}{D}  M \left(M(1-\theta)L\right)^{\frac{1-\theta}{\theta}}>0$ and $b := \frac{1-\theta}{\theta}$.\\
	$\rhd$ If $\theta\in (0,\frac{1}{2}]$, then $b\geq 1$ and $$(R_{t-1}-R_t)^b = O(R_{t-1}-R_t), \text{ as } t\to \infty.$$
	Thus, from \eqref{eq:rec_St_new}, we have
	$$R_t \leq O(R_{t-1}-R_t), \text{ as } t\to \infty.$$
	Hence, $\exists C > 0$ such that 
	$$R_t \leq C(R_{t-1}-R_t), \text{ as } t\to \infty,$$
	implying that $\exists q:=\frac{C}{1+C} \in (0,1)$ such that
	\begin{equation}\label{eq:caseI_theta<0.5}
		\boxed{R_t \leq O\left(q^t\right), \text{ as } t\to \infty.}
	\end{equation}
	$\rhd$ If $\theta\in(\frac{1}{2},1)$, then $b\in (0,1)$ and 
	$$R_{t-1}-R_t = O((R_{t-1}-R_t)^b), \text{ as } t\to \infty.$$
	From \eqref{eq:rec_St_new}, we have
	$$R_t \leq O((R_{t-1}-R_t)^b), \text{ as } t\to \infty.$$
	Hence, $\exists C > 0$ such that 
	$$R_t^{\frac{1}{b}} \leq C(R_{t-1}-R_t), \text{ as } t\to \infty,$$
	with $\frac{1}{b} > 1$. Using Lemma \ref{lem:cvrate-bis} (iii), we get
	$$R_t\leq O\left(t^{\frac{b}{b-1}}\right) = O\left(t^{\frac{1-\theta}{1-2\theta}}\right), \text{ as } t\to \infty.$$
	Hence
	\begin{equation}\label{eq:caseI_theta>0.5}
		\boxed{R_t\leq O\left(t^{\frac{1-\theta}{1-2\theta}}\right), \text{ as } t\to \infty.}
	\end{equation}
	Combining \eqref{eq:ineqresvsSt} ($\|x^t - x^\star\|\leq  R_t$) with \eqref{eq:caseI_theta<0.5} and \eqref{eq:caseI_theta>0.5}, we get that $\{\|x^k-x^\star\|\}$ converges to $0$ linearly if $\theta\in (0,\frac{1}{2}]$ and sublinearly if $\theta \in (\frac{1}{2},1)$.
\end{proof}

\subsubsection{A PL Specialization and Structure-Dependent Rates}\label{subsec:PLspecialization}

The KL-based analysis above also covers stronger error-bound regimes once the Lyapunov function becomes differentiable near the cluster set. This yields a simple answer to whether the KL-dependent part of the theory extends to the Polyak--{\L}ojasiewicz (PL) setting.

\begin{proposition}[PL implies KL with exponent $1/2$]\label{prop:PL-implies-KL}
	
	Let $\Psi:\R^n\to\R_+$ be of class $C^1$ on a neighborhood $\mathcal U$ of $\omega(x^0)$ and suppose that there exists $\mu>0$ such that
	\begin{equation}\label{eq:PL-condition}
		\frac12 \|\nabla \Psi(x)\|^2 \ge \mu \Psi(x), \quad \forall x\in \mathcal U.
	\end{equation}
	Then $\Psi$ satisfies the KL property on $\mathcal U$ with desingularizing function
	\begin{equation}\label{eq:PL-desing}
		\varphi(s)=\sqrt{\frac{2}{\mu}}\,s^{1/2},
	\end{equation}
	and hence with {\L}ojasiewicz exponent $\theta=\frac12$.
	
\end{proposition}

\begin{proof}
	
	Fix $x\in \mathcal U$ with $\Psi(x)>0$. The PL inequality \eqref{eq:PL-condition} gives
	\[
	\|\nabla \Psi(x)\| \ge \sqrt{2\mu}\,\Psi(x)^{1/2}.
	\]
	Since $\Psi$ is $C^1$ on $\mathcal U$, we have $\partial^L \Psi(x)=\{\nabla \Psi(x)\}$ and therefore
	\[
	\dist(0,\partial^L \Psi(x)) = \|\nabla \Psi(x)\|.
	\]
	For $\varphi$ defined by \eqref{eq:PL-desing}, we have
	\[
	\varphi'(s) = \frac{1}{\sqrt{2\mu}}\,s^{-1/2}\qquad (s>0).
	\]
	Hence
	\[
	\varphi'(\Psi(x))\,\dist(0,\partial^L \Psi(x))
	= \frac{\|\nabla \Psi(x)\|}{\sqrt{2\mu}\,\Psi(x)^{1/2}}
	\ge 1,
	\]
	which is exactly the KL inequality. The exponent is therefore $\theta=\frac12$.
	
\end{proof}

\begin{corollary}[Linear convergence in a PL regime]
	
	Under the assumptions of Theorem \ref{thm:globalcv-KL}, assume in addition that the Lyapunov function $\Psi$ is $C^1$ on a neighborhood of $\omega(x^0)$ and satisfies the PL inequality \eqref{eq:PL-condition}. Then the KL exponent in Theorems \ref{thm:cvrate-DCA-KL-fk} and \ref{thm:cvrate-DCA-KL-xk} can be taken as $\theta=\frac12$. In particular, the objective values and the iterates converge linearly whenever the conclusions of those two theorems apply.
	
\end{corollary}

\begin{proof}
	
	By Proposition \ref{prop:PL-implies-KL}, the PL inequality \eqref{eq:PL-condition} implies that $\Psi$ satisfies the KL property with exponent $\theta=\frac12$. Therefore, Theorems \ref{thm:cvrate-DCA-KL-fk} and \ref{thm:cvrate-DCA-KL-xk} apply with $\theta=\frac12$. Since $\theta=\frac12$ falls into the linear regime in both theorems, we obtain linear convergence of the objective values and of the iterates.
	
\end{proof}

\begin{remark}
	
	This should be interpreted as a specialization rather than an equivalence: in the differentiable setting, a PL inequality is a stronger, typically more quantitative error-bound condition than KL$(1/2)$. The point for the present paper is that it fits seamlessly into the same DCA-specific framework and sharpens the rate statement without any appeal to subanalyticity.
	
\end{remark}

\begin{remark}
	
	Theorems \ref{thm:cvrate-DCA-KL-fk} and \ref{thm:cvrate-DCA-KL-xk} also clarify how sharper rates arise for structured application models. Once the model structure yields an explicit KL exponent or a PL/error-bound constant, the rate follows immediately from the same proof template. In this sense, the framework separates the DCA-specific verification of descent and residual bounds from the model-specific task of quantifying regularity.
	
\end{remark}

\appendix
\section{Useful Lemmas}\label{secA1}

The following lemmas are essential for establishing the convergence rate of DCA in our analysis. Beforehand, we recall some classical convergence regimes for a nonnegative sequence $\{r_k\}$:
\begin{itemize}
	\item If $r_k \le c\,k^{-p}$ with $p>0$ and $c>0$, then $\{r_k\}$ converges sublinearly, and achieving an $\varepsilon$-accurate solution requires $O(\varepsilon^{-1/p})$ iterations.
	\item If $r_k \le c\,q^k$ with $0<q<1$ and $c>0$, then $\{r_k\}$ converges linearly, with iteration complexity $O(\ln(\varepsilon^{-1}))$.
	\item If $r_{k+1} \le c\,r_k^2$ with $0 < c r_0 < 1$, then $\{r_k\}$ converges quadratically, with iteration complexity $O(\ln\ln(\varepsilon^{-1}))$.
\end{itemize}

\begin{lemma}\label{lem:cvrate-bis}
	Let $\{r_k\}$ be a nonincreasing and nonnegative sequence converging to $0$. Suppose there exist two positive constants $\alpha$ and $\beta$ such that for all sufficiently large $k$, we have
	\begin{equation}
		\label{eq:seqrel-bis}
		r_{k+1}^{\alpha} \leq \beta (r_k - r_{k+1}).
	\end{equation}
	Then:
	\begin{enumerate}
		\item[(i)] If $\alpha=0$, the sequence $\{r_k\}$ converges to $0$ in a finite number of steps.
		\item[(ii)] If $\alpha \in (0,1]$ and $r_k > 0$ for all $k \in \N$, then 
		$$r_k \leq O\left(\left(\frac{\beta}{1+\beta}\right)^k\right), \text{ as } k \to \infty;$$
		i.e., the sequence $\{r_k\}$ converges linearly to $0$ with rate $\frac{\beta}{1+\beta}$.
		\item[(iii)] If $\alpha > 1$ and $r_k > 0$ for all $k \in \N$, then 
		$$r_k \leq O\left(k^{\frac{1}{1-\alpha}}\right), \text{ as } k \to \infty;$$
		i.e., the sequence $\{r_k\}$ converges sublinearly to $0$.
	\end{enumerate}
\end{lemma}

\begin{proof}
	(i) If $\alpha=0$, then \eqref{eq:seqrel-bis} implies that for large enough $k$ (i.e., there exists $N>0$, $\forall k\geq N$), we have 
	$$0\leq r_{k+1}\leq r_k - \frac{1}{\beta}.$$
	It follows by $r_k\to 0$ and $\frac{1}{\beta}>0$ that $\{r_k\}$ converges to $0$ in a finite number of steps, and we can estimate the number of steps as:
	$$0\leq r_{k+1} \leq r_{k}-\frac{1}{\beta} \leq r_{k-1}-\frac{2}{\beta}\leq \cdots \leq r_{N}-\frac{k-N+1}{\beta}.$$
	Hence $$k\leq \beta r_N + N - 1.$$
	(ii) If $\alpha \in ]0,1]$ and $r_k>0, \forall k\in \N$. Since $r_k\to 0$, we have that $r_k<1$ for large enough $k$. Thus, $r_{k+1}\leq r_k<1$, and it follows by \eqref{eq:seqrel-bis} that 
	$$r_{k+1} \leq r_{k+1}^{\alpha} \leq \beta (r_k-r_{k+1})$$
	for large enough $k$. Hence there exists $N>0$ such that $\forall k\geq N$
	$$r_{k+1}\leq \left(\frac{\beta}{1+\beta}\right) r_k.$$
	So that 
	$$r_k \leq \left(\frac{\beta}{1+\beta}\right)^{k-N} r_N = O\left((\frac{\beta}{1+\beta})^{k}\right), \text{ as } k\to \infty.$$
	That is, $\{r_k\}$ converges linearly to $0$ with rate $\frac{\beta}{1+\beta}$ for large enough $k$.\\
	(iii) If $\alpha>1$ and $r_k>0$ for all $k\in \N$. Let $\phi(t)=t^{-\alpha}$ and $\tau > 1$.\\
	$\rhd$ Suppose that $\phi(r_{k+1})\leq \tau \phi(r_k)$. By the decreasing of $\phi(t)$ and $r_k\geq r_{k+1}$, we have 
	$$\phi(r_k)(r_k-r_{k+1}) \leq \int_{r_{k+1}}^{r_k} \phi(t)~\mathrm{d} t = \frac{1}{1-\alpha}(r_{k}^{1-\alpha}-r_{k+1}^{1-\alpha}).$$
	It follows from \eqref{eq:seqrel-bis} that 
	$$\frac{1}{\beta} \leq \phi(r_{k+1})(r_k-r_{k+1}) \leq \tau\phi(r_{k})(r_k-r_{k+1})  \leq \frac{\tau}{\alpha-1}(r_{k+1}^{1-\alpha} - r_{k}^{1-\alpha}).$$
	Hence
	\begin{equation}
		\label{eq:recineq01}
		r_{k+1}^{1-\alpha} - r_{k}^{1-\alpha} \geq \frac{\alpha-1}{\beta\tau}.
	\end{equation}
	$\rhd$ Suppose that $\phi(r_{k+1})\geq \tau \phi(r_k)$. Taking $q:=\tau^{-\alpha^{-1}}\in (0,1)$, then
	$$r_{k+1} \leq q r_k.$$
	Hence $$r_{k+1}^{1-\alpha} \geq q^{1-\alpha} r_k^{1-\alpha}.$$
	It follows that $\exists N>0, \forall k\geq N$:
	$$r_{k+1}^{1-\alpha} - r_{k}^{1-\alpha} \geq (q^{1-\alpha}-1) r_k^{1-\alpha} \geq (q^{1-\alpha}-1) r_N^{1-\alpha}.$$
	In both cases, there exists a constant $\zeta := \min \{(q^{1-\alpha}-1) r_N^{1-\alpha}, \frac{\alpha-1}{\beta\tau}\}$ such that for large enough $k$, we have 
	$$r_{k+1}^{1-\alpha} - r_{k}^{1-\alpha} \geq \zeta.$$
	Summing up for $k$ from $N$ to $M-1\geq N$, we have 
	$$r_{M}^{1-\alpha} - r_{N}^{1-\alpha} \geq \zeta (M-N).$$
	Then,
	$$r_{M} \leq \left( r_{N}^{1-\alpha} +\zeta (M-N) \right)^{\frac{1}{1-\alpha}} = O\left(M ^{\frac{1}{1-\alpha}}\right).$$
	Hence, 
	$$r_{k} \leq O\left(k^{\frac{1}{1-\alpha}}\right), \text{ as } k\to \infty,$$
	i.e., there exists some $\eta>0$ such that 
	$$r_{k}\leq \eta k^{\frac{1}{1-\alpha}}$$
	for large enough $k$, which completes the proof.
\end{proof}

A similar result to Lemma \ref{lem:cvrate} for the inequality \( r_{k}^{\alpha} \leq \beta (r_k - r_{k+1}) \) is stated below.

\begin{lemma}\label{lem:cvrate}
	Let $\{r_k\}$ be a nonincreasing and nonnegative sequence converging to $0$. Suppose there exist two positive constants $\alpha$ and $\beta$ such that for all sufficiently large $k$, we have 
	\begin{equation}
		\label{eq:seqrel}
		r_k^{\alpha} \leq \beta (r_k - r_{k+1}).
	\end{equation}
	Then:
	\begin{enumerate}
		\item[(i)] If $\alpha=0$, the sequence $\{r_k\}$ converges to $0$ in finitely many iterations.
		\item[(ii)] If $\alpha \in (0,1]$ and $r_k > 0$ for all $k \in \N$, then 
		$$r_k \leq O\left(\left(1-\frac{1}{\beta}\right)^k\right), \text{ as } k \to \infty;$$
		i.e., the sequence $\{r_k\}$ converges linearly to $0$ with rate \( 1-\frac{1}{\beta} \).
		\item[(iii)] If $\alpha > 1$ and $r_k > 0$ for all $k \in \N$, then 
		$$r_k \leq O\left(k^{\frac{1}{1-\alpha}}\right), \text{ as } k \to \infty;$$
		i.e., the sequence $\{r_k\}$ converges sublinearly to $0$.
	\end{enumerate}
\end{lemma}

\begin{proof}
	(i) If $\alpha=0$, then \eqref{eq:seqrel} implies that for large enough $k$ (i.e., there exists $N>0$, $\forall k\geq N$), we have 
	$$0\leq r_{k+1}\leq r_k - \frac{1}{\beta}.$$
	It follows by $r_k\to 0$ and $\frac{1}{\beta}>0$ that $\{r_k\}$ converges to $0$ in a finite number of steps, and we can estimate the number of steps as:
	$$0\leq r_{k+1} \leq r_{k}-\frac{1}{\beta} \leq r_{k-1}-\frac{2}{\beta}\leq \cdots \leq r_{N}-\frac{k-N+1}{\beta}.$$
	Hence $$k\leq \beta r_N + N - 1.$$
	(ii) If $\alpha \in ]0,1]$ and $r_k>0, \forall k\in \N$, then we get from $r_k\to 0$ that $r_k<1$ for large enough $k$. Thus, it follows by \eqref{eq:seqrel} that 
	$$r_k \leq r_k^{\alpha} \leq \beta (r_k-r_{k+1})$$
	for large enough $k$ ($k\geq N$). Hence
	$$r_{k+1}\leq (1-\frac{1}{\beta}) r_k, \text{ with } \beta>1,$$
	so that $$r_k\leq (1-\frac{1}{\beta})^{k-N} r_N = O((1-\frac{1}{\beta})^{k}), \text{ as } k\to \infty,$$
	i.e., $\{r_k\}$ converges linearly to $0$ with rate $1-\frac{1}{\beta}$ for large enough $k$.\\
	(iii) If $\alpha>1$ and $r_k>0$ for all $k\in \N$. By the decreasing of $\phi(t)=t^{-\alpha}$ and $r_k\geq r_{k+1}$, we have 
	$$\phi(r_k)(r_k-r_{k+1}) \leq \int_{r_{k+1}}^{r_k} \phi(t)~\mathrm{d} t.$$
	Then
	$$\frac{1}{\beta} \overset{\eqref{eq:seqrel}}{\leq}\phi(r_k)(r_k-r_{k+1}) \leq \int_{r_{k+1}}^{r_k} \phi(t)~\mathrm{d} t = \frac{r_{k+1}^{1-\alpha} - r_{k}^{1-\alpha}}{\alpha-1}.$$
	Hence,
	$$r_{k+1}^{1-\alpha} - r_{k}^{1-\alpha} \geq \frac{\alpha-1}{\beta}, \forall k\geq N.$$
	Summing up for $k$ from $N$ to $M-1 (\geq N)$, we have 
	$$r_{M}^{1-\alpha} - r_{N}^{1-\alpha} \geq \frac{\alpha-1}{\beta} (M-N).$$
	Then,
	$$r_{M} \leq \left( r_{N}^{1-\alpha} + \frac{\alpha-1}{\beta} (M-N) \right)^{\frac{1}{1-\alpha}} = O(M ^{\frac{1}{1-\alpha}}).$$
	Hence, 
	$$r_{k}\leq O(k^{\frac{1}{1-\alpha}}), \text{ as } k\to \infty.$$
	That is, there exists some $\eta>0$ such that 
	$$r_{k}\leq \eta k^{\frac{1}{1-\alpha}}$$
	for large enough $k$, which completes the proof. 
\end{proof}

\begin{lemma}\label{lem:cvrate-tri}
	Let $\{r_k\}$ be a nonincreasing and nonnegative sequence converging to $0$. Suppose there exist three constants $a > 0$, $b > 0$, and $c > 0$ such that for sufficiently large $k$, we have 
	\begin{equation}
		\label{eq:seqrel-tri}
		r_{k} \leq c (r_{k-1} - r_{k}) + a (r_k - r_{k+1})^b.
	\end{equation}
	Then:
	\begin{enumerate}
		\item[(i)] If \( b \geq 1 \), then \( \exists q \in (0,1) \) such that 
		$$r_k \leq O(q^k) \text{ as } k \to \infty;$$
		i.e., the sequence $\{r_k\}$ converges linearly. 
		\item[(ii)] If \( b \in (0,1) \), then     
		$$r_k \leq O\left( k^{\frac{b}{b-1}} \right) \text{ as } k \to \infty;$$
		i.e., the sequence $\{r_k\}$ converges sublinearly.
	\end{enumerate}
\end{lemma}

\begin{proof}
	The basic idea is to reduce the two terms on the right hand side to one term by asymptotic behavior of the residuals.\\
	(i) If $b\geq 1$, as $r_k-r_{k+1}\to 0$, then for large enough $k$, we have 
	$$(r_k - r_{k+1})^b \leq r_k - r_{k+1}.$$
	Then \eqref{eq:seqrel-tri} is reduced to 
	$$r_{k} \leq c (r_{k-1}-r_{k}) + a (r_k - r_{k+1}) \text{ as } k \to \infty.$$
	Now, consider two cases: \\
	$\rhd$ If $r_{k-1}-r_{k}\sim_{k\to \infty} O(r_{k}-r_{k+1})$ (clearly including the case $r_{k-1}-r_{k}\sim_{k\to \infty} r_{k}-r_{k+1}$), then 
	$$r_k \leq  O(r_k - r_{k+1}) + a (r_k - r_{k+1}) = O(r_k - r_{k+1}) \text{ as } k\to \infty.$$
	Hence, there exists $\beta$ large enough (let $\beta>1$) such that
	$$r_k \leq \beta (r_k - r_{k+1}) \text{ as } k \to \infty.$$
	Then, we get :
	$$r_{k+1} \leq \frac{\beta-1}{\beta} r_k \text{ as } k\to \infty.$$
	Hence, $\exists q := \frac{\beta-1}{\beta}\in (0,1)$ such that
	$$\boxed{r_k\leq O(q^k) \text{ as } k\to \infty.}$$
	$\rhd$ If $r_{k}-r_{k+1}\sim_{k\to \infty} O(r_{k-1}-r_{k})$, then we get in a similar way that 
	$\exists \beta > 0$ such that
	$$r_k \leq \beta (r_{k-1} - r_{k}) \text{ as } k\to \infty.$$
	Hence, there exists $q:=\frac{\beta}{1+\beta}\in (0,1)$ such that
	$$\boxed{r_k\leq O(q^k) \text{ as }k\to \infty.}$$
	In both cases, we have the linear convergence of the sequence $\{r_k\}$.\\
	(ii) If $b\in (0,1)$, we can use a similar technique to simplify the right hand side. \\
	$\rhd$ If $r_{k-1}-r_{k}\sim_{k\to \infty} O(r_{k}-r_{k+1})$, then 
	$$r_k \leq a(r_k-r_{k+1})^b + O((r_{k}-r_{k+1})^b) \text{ as } k\to \infty.$$
	Hence, there exists $C>0$ such that 
	$$r_k \leq C(r_k-r_{k+1})^b,$$
	that is 
	$$r_k^{\frac{1}{b}} \leq C^{\frac{1}{b}}(r_k-r_{k+1}).$$
	Thus, $\exists \alpha := \frac{1}{b}>1, \beta:=C^{\frac{1}{b}} > 0$ such that
	$$\boxed{r_k^{\alpha} \leq \beta (r_k - r_{k+1}) \text{ as } k\to \infty.}$$
	$\rhd$ If $r_{k}-r_{k+1}\sim_{k\to \infty} O(r_{k-1}-r_{k})$, then 
	$$r_k \leq O((r_{k-1}-r_{k})^b) \text{ as } k\to \infty.$$
	Hence, $\exists \alpha := \frac{1}{b}>1, \beta > 0$ such that
	$$\boxed{r_k^{\alpha} \leq \beta (r_{k-1} - r_{k}) \text{ as } k\to \infty.}$$
	In both cases, we apply Lemmas \ref{lem:cvrate} (iii) and \ref{lem:cvrate-bis} (iii) respectively to prove that $$\boxed{r_k \leq O(k^{\frac{1}{1-\alpha}}) = O\left( k^{\frac{b}{b-1}} \right) \text{ as } k\to \infty.}$$
	Hence, the sequence $\{r_k\}$ converges sublinearly.
\end{proof}

\section*{Acknowledgements}

The author thanks the editor and the anonymous reviewers for their careful reading and constructive suggestions, which helped improve the presentation of this paper. This work was supported by the National Natural Science Foundation of China (Grant Nos.\ 42450242 and 11601327), the Beijing Overseas High-Level Talent Program, and institutional support from the Beijing Institute of Mathematical Sciences and Applications (BIMSA).

\bibliography{references}

\end{document}